\documentclass[11pt]{article}
\usepackage{color,latexsym,amsthm,amsmath,amssymb,path,enumitem,hyperref,tikz,geometry,subcaption}
\usepackage[font={small,sf}]{caption}
\newcommand{\ignore}[1]{}

\makeatletter
\def\th@plain{%
  \thm@notefont{}
  \itshape 
}
\def\th@definition{%
  \thm@notefont{}
  \normalfont 
}
\makeatother

\newtheorem{theorem}{Theorem}[section]
\newtheorem{lemma}[theorem]{Lemma}
\newtheorem{corollary}[theorem]{Corollary}

\usepackage{graphicx,psfrag}

\usepackage{listings,xcolor}
\usepackage{dsfont}

\newcommand{\RR}{\ensuremath{\mathbb R}}

\newcommand{\PP}{\ensuremath{\mathbb P}}

\newcommand{\curves}{\mathcal C}
\newcommand{\family}{\mathcal F}
\newcommand{\circs}{\Gamma}
\newcommand{\lsets}{\mathcal S}

\newcommand{\pts}{\mathcal P}
\newcommand{\vars}{\mathcal V}
\newcommand{\vb}{{\bf V}}
\newcommand{\crc}{C}
\newcommand{\planes}{\Pi}
\newcommand{\crs}{\mathrm{cr}}

\def\eps{{\varepsilon}}

\newcommand{\parag}[1]{\vspace{2mm}

\noindent{\bf #1} }

\title{Distinct distances on non-ruled surfaces and between circles\thanks{This research project was done as part of the 2019 CUNY Combinatorics REU, supported by NSF awards DMS-1802059 and DMS-1851420.}}

\author{
Surya Mathialagan\thanks{Department of Electrical Engineering and Computer Science, MIT, MA, USA.
{\sl smathi@mit.edu.} This research was completed while the author was at Caltech. Supported by Caltech's Summer Undergraduate Research Fellowships (SURF) Program and the Olga Taussky-Todd Award.}
\and
Adam Sheffer\thanks{Department of Mathematics, Baruch College, City University of New York, NY, USA.
{\sl adamsh@gmail.com}. Supported by NSF award DMS-1802059.}}

\begin{document}

\pagenumbering{arabic}

\maketitle
\begin{abstract}
We improve the current best bound for distinct distances on non-ruled algebraic surfaces in $\RR^3$. In particular, we show that $n$ points on such a surface span $\Omega\left(n^{32/39-\eps}\right)$ distinct distances, for any $\eps>0$. 
Our proof adapts the proof of Sz\'ekely for the planar case, which is based on the crossing lemma.

As part of our proof for distinct distances on surfaces, we also obtain new results for distinct distances between circles in $\RR^3$. Consider two point sets of respective sizes $m$ and $n$, such that each set lies on a distinct circle in $\RR^3$. We characterize the cases when the number of distinct distances between the two sets can be $O(m+n)$. This includes a new configuration with a small number of distances. In any other case, we prove that the number of distinct distances is $\Omega\left(\min\left\{m^{2/3}n^{2/3},m^2,n^2\right\}\right)$.
\end{abstract}

\section{Introduction}

Erd\H os introduced the family of \emph{distinct distances problems} in 1946 and considered it to be his ``most striking contribution to geometry'' \cite{Erd96}. For over 50 years, he regularly added more conjectures and problems to this family. For a finite point set $\pts\subset \RR^d$, let $\Delta(\pts)$ be the set of distances spanned by pairs of points from $\pts$. 
In his first seminal paper on the subject, Erd\H os \cite{erd46} asked for the asymptotic value of $D(\pts) =\min_{|\pts|=n} |\Delta(\pts)|$.
In other words, the problem asks for the minimum number of distances that could be spanned by $n$ points in $\RR^d$.

The answer to the above problem depends on $d$.
In $\RR^2$, Erd\H os presented a set of $n$ points that spans $\Theta(n/\sqrt{\log n})$ distinct distances and conjectured that this was asymptotically optimal.
Over the decades, a large number of works have been dedicated to this problem (such as \cite{Chung84,SolyToth01,Szek97,Tardos03}). 
Recently, Guth and Katz \cite{GK15} almost completely settled Erd\H os's conjecture, proving that $D(\pts)=\Omega(n/\log n)$ for any set $\pts\subset\RR^2$ of $n$ points. 

\parag{Distinct distances on non-ruled surfaces.} When the breakthrough of Guth and Katz first appeared in 2010, the community was optimistic about using a  similar technique to solve the distinct distances problem in every dimension. 
This turned out to be more difficult than expected, and so far no progress has been made even in $\RR^3$. 
One step towards solving the problem in $\RR^3$ may be to consider the problem on an algebraic surface.
That is, we consider the number of distinct distances spanned by $n$ points on a specific surface in $\RR^3$.

The case of $n$ points on a plane in $\RR^3$ is equivalent to the distinct distances problem in $\RR^2$.
Tao \cite{Tao11} showed that the Guth-Katz result can be extended in the following way. 

\begin{theorem} \label{th:Tao}
Let $U\subset \RR^3$ be a sphere or a hyperboloid of two sheets. 
Then any set $\pts$ of $n$ points on $U$ satisfies 
\[ D(\pts) = \Omega\left(n/\log n\right). \]
\end{theorem}

No known constructions of $n$ points on an algebraic surface in $\RR^3$ span a sublinear number of distinct distances (with the obvious exception of points on a plane).  
On the other hand, there are unrestricted point sets $\pts\subset\RR^3$ that satisfy $D(\pts)=\Theta(n^{2/3})$.
(It is conjectured that no set in $\RR^3$ spans an asymptotically smaller number of distinct distances.)

Sharir and Solomon \cite{ShaSo17} proved the following distinct distances result on surfaces in $\RR^3$. (The definition of an irreducible two-dimensional variety of degree $k$ can be found in Section \ref{sec:prelim}.) 

\begin{theorem}\label{th:SharirSolomon}
Let $U\subset \RR^3$ be an irreducible two-dimensional variety of degree $k$.
Let $\pts$ be a set of $n$ points on $U$. 
Then for any $\eps>0$, we have 
\[ D(\pts) = \Omega_k\left(n^{7/9-\eps}\right)\footnote{We use $O_{v_1, v_2, \dots, v_k}$ to represent the usual big-$O$ notation where the constant of proportionality depends on the variables $v_1, \dots, v_k$. We define $\Omega_{v_1, v_2, \dots, v_k}$ and $\Theta_{v_1, v_2, \dots, v_k}$ symmetrically.}. \]
\end{theorem}
In another work, Sharir and Solomon \cite{ShaSo16} replace the extra $\eps$ in the exponent with a polylogarithmic factor.
These bounds can be thought of as adaptations of Chung's bound for distinct distances in the plane \cite{Chung84}. These results reveal that the three-dimensional distinct distances problem behaves differently when restricting the point set to a surface.

We improve Theorem 1.2 for non-ruled surfaces.
A surface $U\subset\RR^3$ 
is \emph{ruled} if for every point $p\in U$ there exists a line that is contained in $U$ and incident to $p$.
For example, cylindrical and conical surfaces are ruled.  
\begin{theorem} \label{th:SurfDD}
Let $U\subset \RR^3$ be an irreducible non-ruled two-dimensional variety of degree $k$.
Let $\pts$ be a set of $n$ points on $U$. 
Then for any $\eps>0$, we have 
\[ D(\pts) = \Omega_k\left(n^{32/39-\eps}\right). \]
\end{theorem}
The first step of our proof extends Sz\'ekely's technique for distinct distances in the plane \cite{Szek97}. By combining this technique with incidence bounds for curves, we obtain the bound $\Omega(n^{4/5-\eps})$.
We further improve this bound by studying perpendicular bisectors of points on surfaces in $\RR^3$.
We prove Theorem \ref{th:SurfDD} in Section \ref{sec:DDsurf}.

While Sharir and Solomon do not state this, their proof of Theorem \ref{th:SharirSolomon} also implies that there exists a point $p \in \pts$ that spans $\Omega_k(n^{7/9 - \epsilon})$ distances with $\pts\setminus\{p\}$.
While this is also the case for Sz\'ekely's proof in $\RR^2$, our proof of Theorem \ref{th:SurfDD} does not immediately show the existence of such a point.
Some of the new components that extend Sz\'ekely's proof to $\RR^3$ do not easily extend in this direction. 

\parag{Distinct distances between two circles in $\RR^3$.} As part of the proof of Theorem \ref{th:SurfDD}, we derive a bound on the number of distinct distances between points on two circles in $\RR^3$.
We believe that this result is also intrinsically interesting, and not just with respect to Theorem \ref{th:SurfDD}.

Consider two finite sets $\pts_1,\pts_2\subset \RR^d$.
We denote by $D(\pts_1,\pts_2)$ the number of distinct distances spanned by pairs in $\pts_1\times\pts_2$. 
In other words, we ignore distances between pairs of points from the same set.
Many such bipartite distinct distances problems have been studied (for example, see \cite{Elekes95,PdZ13,SSS13}). We are interested in the case where there exist circles $C_1$ and $C_2$ such that $\pts_1\subset C_1$ and $\pts_2\subset C_2$.

\begin{figure}[ht]
    \centering
    \begin{tikzpicture}[scale=0.5]
        \draw (0, 0) circle(3.5);
        \draw (0, 0) circle(2);
        \foreach \r in {0, 45, ..., 335} {
            \draw[dashed] (\r:3.5) -- ({\r + 45}: 3.5); 
            \draw[dashed] (\r:2) -- ({\r + 45}: 2); 
            \draw[fill=red] (\r: 3.5) circle(1mm); 
            \draw[fill=blue] (\r: 2) circle(1mm);} 
    \end{tikzpicture}
\caption{The case of two concentric circles.}
\label{fi:concentric}
\vspace{-2mm}
\end{figure}
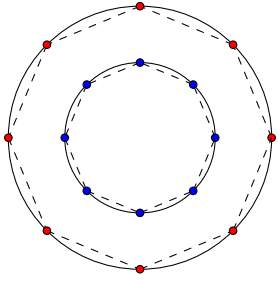

We begin with the case where $C_1$ and $C_2$ are concentric circles in $\RR^2$. Let $\pts_1$ be the set of vertices of a regular $n$-gon lying on $C_1$.
Let $\pts_2$ be a uniform scaling of $\pts_1$ around the center of the circles, such that $\pts_2\subset C_2$.
The case of $n=8$ is depicted in Figure \ref{fi:concentric}. 
Arbitrarily fix a point $p\in\pts_1$.
By symmetry, we have that \[ D(\pts_1,\pts_2)= D(\{p\},\pts_2) = \Theta(n). \]

It is not difficult to generalize the above construction to the case where $\pts_1$ and $\pts_2$ are not necessarily of the same size. 
When $|\pts_1|=m$ and $|\pts_2|=n$, we obtain $D(\pts_1,\pts_2) = \Theta(m+n)$. On the other hand,
Pach and de Zeeuw \cite{PdZ13} proved that when $C_1$ and $C_2$ are not concentric, we have
\[ D(\pts_1,\pts_2) = \Omega\left(\min\left\{m^{2/3}n^{2/3},m^2,n^2\right\}\right). \]

In the current work, we consider the case where $C_1$ and $C_2$ are in $\RR^3$.
We define the \emph{axis} of a circle $\crc$ in $\RR^3$  to be the line incident to the center of $\crc$ and orthogonal to the plane containing $\crc$.
Note that every point on the axis of $\crc$ is equidistant from all the points of $\crc$.
On the other hand, a point not on the axis of $\crc$ cannot be equidistant from three points of $\crc$ (every sphere centered at this point intersects $\crc$ in at most two points).
We say that two circles $\crc_1$ and $\crc_2$ in $\RR^3$ are \emph{aligned} if they have the same line as their axis. 
An example is depicted in Figure \ref{fig:TwoCircles}(a).
Note that the planes that contain aligned circles are parallel. 

The above example of concentric circles in $\RR^2$ can be easily extended to the case of aligned circles in $\RR^3$. Thus, when $C_1$ and $C_2$ are aligned, one can find $\pts_1$ and $\pts_2$ such that $D(\pts_1,\pts_2)=\Theta(m+n)$.
Surprisingly, we also discovered a less intuitive family of constructions with a linear number of distances between two circles.
Let $H_1$ be the plane containing $\crc_1$ and let $H_2$ be the plane containing $\crc_2$.
We say that $\crc_1$ and $\crc_2$ are \emph{perpendicular} if all the following hold:
\begin{enumerate}[label=(\arabic*),noitemsep,topsep=1pt]
\item The planes $H_1$ and $H_2$ are perpendicular (that is, the angle between the two planes is $\pi/2$).
\item The center of $\crc_1$ lies on $H_2$.
\item The center of $\crc_2$ lies on $H_1$.
\end{enumerate}
An example is depicted in Figure \ref{fig:TwoCircles}(b).

\begin{figure}[ht]
    \centering
    \begin{subfigure}[b]{0.33\textwidth}
    \centering
        \includegraphics[width=0.97\textwidth]{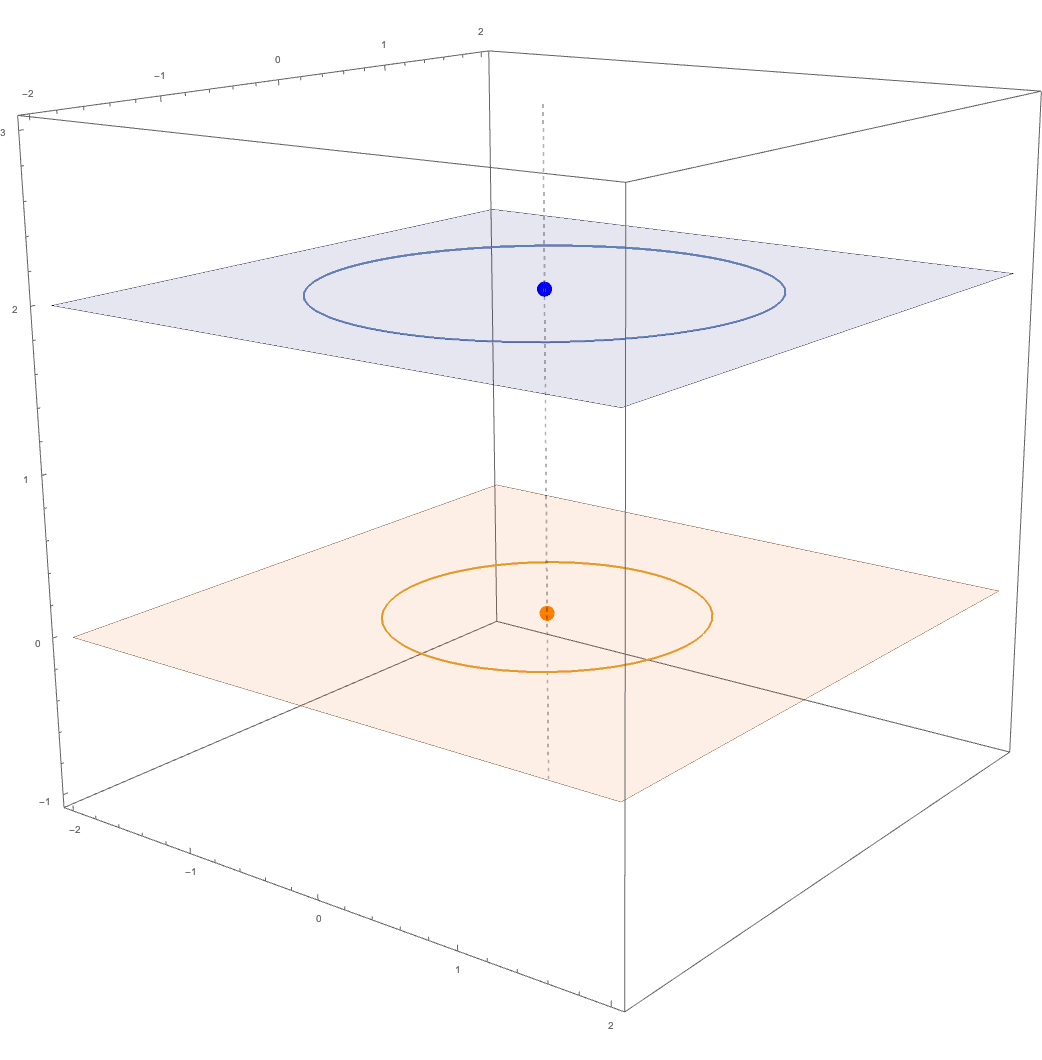}
        \caption{Aligned circles.}\label{subfig:aligned}
    \end{subfigure}
    \hspace{1cm}
    \begin{subfigure}[b]{0.33\textwidth}
        \centering
        \includegraphics[width=\textwidth]{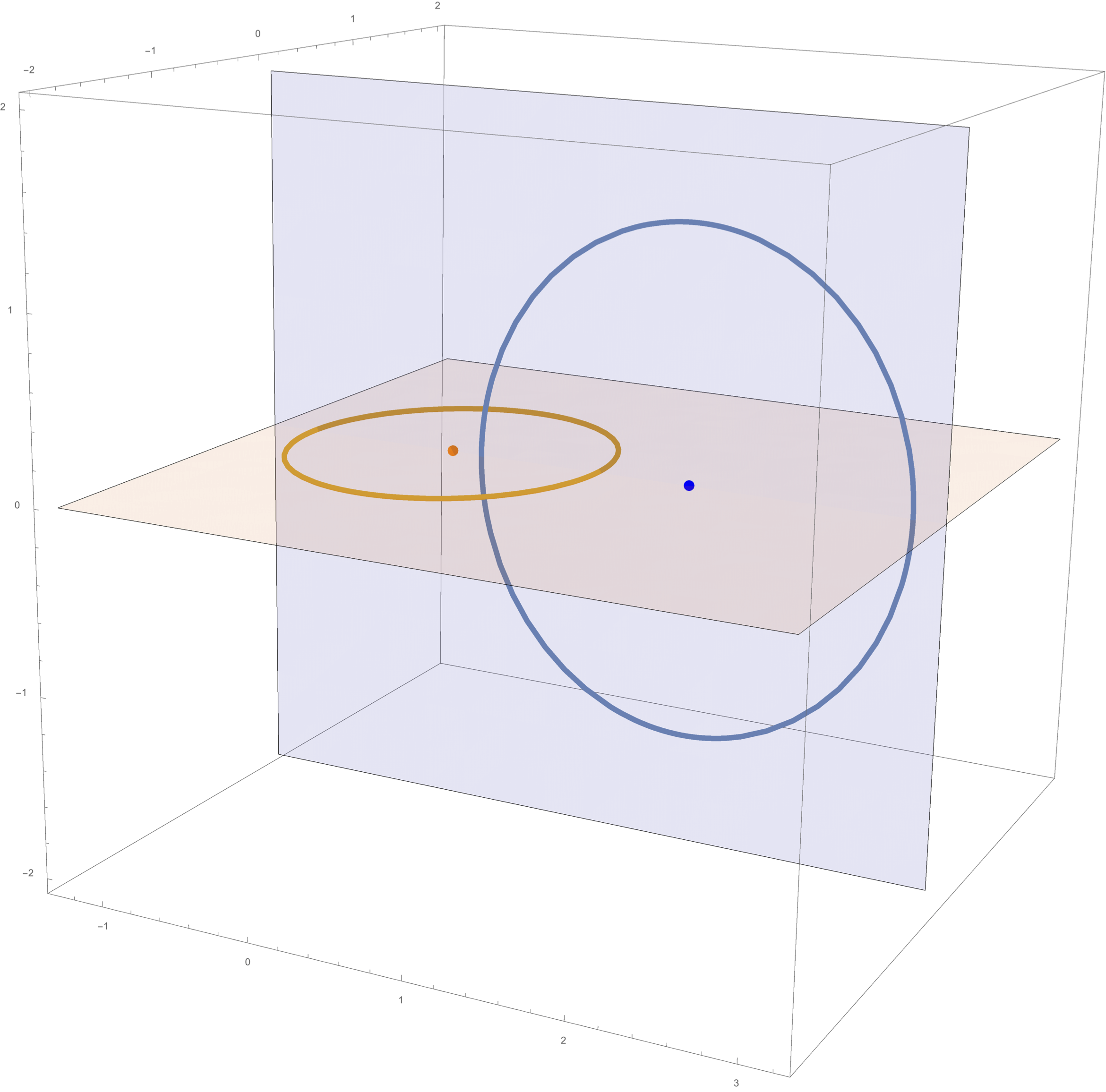}
        \caption{Perpendicular circles.}\label{subfig:perpendicular}
    \end{subfigure}
    \caption{Pairs of circles that may lead to a small number of distinct distances.}
    \label{fig:TwoCircles}
\end{figure}

The following is our main contribution to distinct distances between two circles. 

\begin{theorem} \label{th:CircDD}
Let $\crc_1$ and $\crc_2$ be two circles in $\RR^3$. \\[2mm]
(a) Assume that $\crc_1$ and $\crc_2$ are aligned or perpendicular.
Then there exist a set $\pts_1\subset \crc_1$ of $m$ points and a set $\pts_2\subset \crc_2$ of $n$ points, such that
\[ D(\pts_1,\pts_2) = \Theta(m+n). \]
(b) Assume that $\crc_1$ and $\crc_2$ are neither aligned nor perpendicular.
Let $\pts_1\subset \crc_1$ be a set of $m$ points and let $\pts_2\subset \crc_2$ be a set of $n$ points.
Then
\[ D(\pts_1,\pts_2) = \Omega\left(\min\left\{m^{2/3}n^{2/3},m^2,n^2\right\}\right). \]
\end{theorem}

The proof of Theorem \ref{th:CircDD}(b) relies on Elekes-Szab\'o type expanding polynomials, similar to an argument of Raz \cite{Raz20}. 
This argument is combined with a new layer of computer calculations. 
Part (a) of Theorem \ref{th:CircDD} is proved in Section \ref{sec:ProofA}.
Part (b) is proved in Section \ref{sec:ProofB}. 

\parag{Future directions.}
The main issue of Theorem \ref{th:SurfDD} is that it does not hold for ruled surfaces. It seems plausible that our proof could be extended to also hold for ruled surfaces. This could be an interesting direction for future work.

Another direction to explore is distinct distances between two circles in $\RR^4$. In this case, there exists a simple construction with only one distance between two circles (see \cite{Lenz55}). We wonder what other surprises might exist in $\RR^4$.

\parag{Acknowledgements.}
We are grateful to Frank de Zeeuw for many useful discussions, including help with Theorem \ref{th:RSdZ} and with the circle constructions. 
We thank Toby Aldape, Jingyi (Rose) Liu, Minh-Quan Vo, and the anonymous referees for helping to improve this paper. We also thank Sara Fish for inspirational support. The first author would like to thank everyone involved in the 2019 CUNY REU for motivating her with their passion and engaging with her in many helpful discussions.

\section{Preliminaries} \label{sec:prelim}
We now introduce various definitions and tools that are used in the proofs in the following sections.

We briefly survey notation and results from real algebraic geometry. 
For references and more information, see for example \cite{BCR98,Harris92}.

For polynomials $f_1,\ldots,f_k\in \RR[x_1,\ldots,x_d]$, the \emph{variety} defined by $f_1,\ldots,f_k$ is
\[ \vb(f_1,\ldots,f_k) = \left\{p\in \RR^d\ :\ f_1(p)=f_2(p) = \cdots = f_k(p)=0 \right\}. \]
We say that a set $U \subset \RR^d$ is a variety if there exist $f_1,\ldots,f_k\in \RR[x_1,\ldots,x_d]$ such that $U = \vb(f_1,\ldots,f_k)$.
While not true over some other fields, in $\RR^d$ every variety can be defined using a single polynomial. 

A variety $U$ is \emph{irreducible} if there do not exist two nonempty varieties $U_1,U_2\subset U$ such that $U=U_1 \cup U_2$, $U_1\neq U$, and $U_2\neq U$.
The \emph{dimension} of an irreducible variety $U$ is the largest integer $d_U$ for which there exist non-empty irreducible varieties $U_0,U_1,\ldots,U_{d_U-1}$ such that
\[ U_0 \subsetneq U_1 \subsetneq \ldots \subsetneq U_{d_U-1} \subsetneq U. \]
The dimension of a reducible variety $W$ is the maximum dimension of an irreducible component of $W$. 
We define a \emph{curve} to be an irreducible variety of dimension one. 
A \emph{surface} is an irreducible variety of dimension two.

The \emph{ideal} of a variety $U\subseteq \RR^d$, denoted ${\bf I}(U)$, is the set of polynomials in $\RR[x_1,\ldots,x_d]$ that vanish on every point of $U$.
We say that a set of polynomials $f_1,\ldots,f_\ell \in \RR[x_1,\ldots,x_d]$ \emph{generate} ${\bf I}(U)$ if every element of ${\bf I}(U)$ can be written as $\sum_{j=1}^\ell f_j g_j$ for some $g_1,\ldots,g_\ell \in \RR[x_1,\ldots,x_d]$.
The \emph{Jacobian matrix} of the polynomials $f_1,\ldots,f_k \in \RR[x_1,\ldots,x_d]$ is

\[ {\bf J}_{f_1,\ldots,f_k} = \left( \begin{array}{cccc}
\frac{\partial f_1}{\partial x_1} & \frac{\partial f_1}{\partial x_2} & \cdots & \frac{\partial f_1}{\partial x_d} \\[2mm]
\frac{\partial f_2}{\partial x_1} & \frac{\partial f_2}{\partial x_2} & \cdots & \frac{\partial f_2}{\partial x_d} \\[2mm]
\cdots & \cdots & \cdots & \cdots \\[2mm]
\frac{\partial f_k}{\partial x_1} & \frac{\partial f_k}{\partial x_2} & \cdots & \frac{\partial f_k}{\partial x_d} \end{array}\right)\]

Consider a variety $U\subset\RR^d$ of dimension $k$, and polynomials $f_1,\ldots,f_\ell \in \RR[x_1,\ldots,x_d]$ that generate ${\bf I}(U)$.
We say that $p\in U$ is a \emph{singular} point of $U$ if $\mathrm{rank\,} {\bf J}_{f_1,\ldots,f_k}(p) <d-k$.
A point of $U$ that is not singular is said to be a \emph{regular} point of $U$.
We denote the set of regular points of a variety $U\subset \RR^d$ as $U_\text{reg}$.

We define the \emph{degree} of a surface $U\subset \RR^3$ as the minimum degree of a polynomial $f\in\RR[x,y,z]$ that satisfies $\vb(f)=U$.
There are several non-equivalent definitions for the degree of a variety in $\RR^3$ that is not a surface. 
To avoid this issue, we say that the \emph{complexity} of a variety $U$ is the minimum integer $D$ that satisfies the following. 
There exist $k\le D$ polynomials $f_1,\ldots,f_k\in \RR[x_1,\ldots,x_d]$, each of degree at most $D$, such that $U=\vb(f_1,\ldots,f_k)$. In the past decade, the use of complexity is becoming more common. For example, see \cite{BGT11,SolyTao12}.

\begin{theorem} \label{th:comps}
Let $U\subset \RR^d$ be a variety of complexity $D$.  \\
(a) The number of irreducible components of $U$ is $O_{d,D}(1)$. \\
(b) The number of connected components of $U$ is $O_{d,D}(1)$.
\end{theorem}

For the following quantitative variant of Theorem \ref{th:comps}(b), see Solymosi and Tao \cite{SolyTao12} (see also Barone and Basu \cite{BarBas12}).

\begin{theorem} \label{th:SolyTao}
Let $U\subset \RR^d$ be a variety of dimension $k$ and complexity $C$. Let $f\in \RR[x_1,\ldots,x_d]$ be a polynomial of degree $D$. Then the number of connected components of $U\setminus \vb(f)$ is $O_{C,D,d,k}(D^k)$.
\end{theorem}

For more information about the following lemma, see for example \cite[Section 3.3]{BCR98}.

\begin{lemma} \label{le:singular}
Let $U\subset\RR^3$ be a variety of complexity $D$ and dimension $d>0$.
Then the set of singular points of $U$ is a variety of dimension at most $d-1$ and complexity $O_D(1)$. 
\end{lemma}

Let $S\subset \RR^d$. 
The \emph{Zariski closure} of $S$, denoted $\overline{S}$, is the smallest variety that contains $S$.
Specifically, every variety that contains $S$ also contains $\overline{S}$.
A set $S\subset \RR^d$ is  \emph{semi-algebraic} if there exists a boolean function $\Phi(y_1,\ldots,y_t)$ and polynomials $f_1,\ldots,f_t\in \RR[x_1,\ldots,x_d]$ such that 
\[ p\in S \qquad \text{if and only if } \qquad \Phi(f_1(p)\ge 0,\ldots,f_t(p)\ge 0)=1. \]

The dimension of $S$ is $\dim \overline{S}$.
The \emph{complexity} of $S$ is the minimum $t$ such that $S$ can be described with at most $t$ polynomials of degree at most $t$.
The projection of a real variety may not be a variety. 
However, it must be semi-algebraic. 

\begin{lemma} \label{le:projection}
Let $U \subset \RR^d$ be a variety of complexity $k$ and of dimension $d'$. Let $\pi: \RR^d \to \RR^e$ be a
standard projection: a linear map that keeps $e$ out of the $d$ coordinates of a point in $\RR^d$.
Then $\pi(U)$ is a semi-algebraic set of dimension at most $d'$ and of complexity $O_{k,d}(1)$.
\end{lemma}

For more details about Lemma \ref{le:projection}, see for example \cite[Section 14.2]{BPR06}.
A projection is equivalent to adding an existential quantifier to some of the variables. 
The above reference discusses how to then eliminate such quantifiers.

For more information about the following lemma, see for example \cite[Section 16.4]{BPR06}.

\begin{lemma} \label{le:ConnSA}
Let $S\subset \RR^d$ be a semi-algebraic set of complexity $t$. Then the number of connected components of $S$ is $O_{d,t}(1)$.
\end{lemma}

\parag{Distinct distances: first bounds.} 
We state a couple of simple observations involving distinct distances in $\RR^3$.
For a point $p\in\RR^3$ and $r>0$, we denote by $S(p,r)$ the sphere of radius $r$ centered at $p$.

\begin{lemma} \label{le:DDcurve}
Let $\gamma$ be a one-dimensional variety of  complexity $O(1)$ in $\RR^3$. 
Let $\pts$ be a set of $n$ points on $\gamma$.
Then $D(\pts)=\Omega(n)$.
\end{lemma}
\begin{proof}
By Theorem \ref{th:comps}, $\gamma$ can be partitioned to $O(1)$ irreducible components. 
By the pigeonhole principle, there exists an irreducible one-dimensional component $\gamma'$ of $\gamma$ that contains $\Theta(n)$ points of $\pts$. 
Set $\pts' = \pts \cap \gamma'$.

Let $p$ be an arbitrary point of $\pts'$.
Fix $r>0$, and note that the sphere $S(p,r)$ cannot contain $\gamma'$.
Since $\gamma'$ is irreducible, the intersection $S(p,r)\cap \gamma'$ is of dimension zero (or empty). 
By Theorem \ref{th:comps}, this intersection consists of $O(1)$ points. 
That is, the number of points of $\pts'$ at distance $r$ from $p$ is $O(1)$. 
This implies that 
\[ D(\pts) \ge D(\pts') \ge D(\{p\},\pts'\setminus \{p\})=\Omega(n). \]
\end{proof}

Given two points $a,b\in \RR^d$, we denote the distance between $a$ and $b$ as $|ab|$.
\begin{lemma}\label{le:parallelPlanes}
Let $H_1$ be the $xy$-plane in $\RR^3$ and let $H_2$ be a plane parallel to $H_1$ (possibly $H_1=H_2$).
Let $\pts_1\subset H_1$ and $\pts_2\subset H_2$ be finite point sets.
Let $\pi:\RR^3 \to H_1$ be the projection obtained by setting the $z$-coordinate to zero.
Then 
\[ D(\pts_1,\pts_2) = D(\pts_1,\pi(\pts_2)). \]
\end{lemma}
\begin{proof}
The plane $H_2$ is defined by an equation of the form $z=c$.
The distance between points $(a_x,a_y,0)\in H_1$ and $(b_x,b_y,c)\in H_2$ is
\[ \sqrt{(a_x-b_x)^2+(a_y-b_y)^2+c^2}.\]
The distance between $(a_x,a_y,0)\in H_1$ and $\pi(b_x,b_y,c) = (b_x,b_y,0)$ is
\[ \sqrt{(a_x-b_x)^2+(a_y-b_y)^2}.\]

Let $a,a'\in \pts_1$ and $b,b'\in \pts_2$.
By the above, we note that $|ab|=|a'b'|$ if and only if $|a\pi(b)|=|a\pi(b')|$.
Indeed, the square of both distances changes by $c^2$.
We conclude that  $D(\pts_1,\pts_2) = D(\pts_1,\pi(\pts_2))$.
\end{proof}

\section{Distinct distances on non-ruled surfaces} \label{sec:DDsurf}

In this section we  prove Theorem \ref{th:SurfDD}.
We first present some additional preliminaries in Section \ref{ssec:AdditionalPrelim}. In Section \ref{ssec:Surfaces}, we study some properties of surfaces in $\RR^3$.
Finally, we prove Theorem \ref{th:SurfDD} in Sections \ref{ssec:SurfDDproof} and \ref{ssec:RichEdgesLemma}.

\subsection{Additional preliminaries} \label{ssec:AdditionalPrelim}

\parag{Incidences.}
Let $\pts$ be a set of points and let $\vars$ be a set of varieties, both in $\RR^d$.
An incidence is a pair $(p,U)\in\pts\times \vars$ such that the point $p$ lies on the variety $U$. 
We denote the number of incidences in $\pts\times \vars$ as $I(\pts,\vars)$.

A polynomial in $\RR[x,y]$ of degree at most $k$ has at most $\binom{k+2}{2}$ monomials. 
We can thus define every such polynomial by a set of $\binom{k+2}{2}$ real coefficients.
This leads to a bijection between the set of such polynomials and $\RR^{\binom{k+2}{2}}$.
Since we are only interested in the zero sets of the polynomials, we think of two polynomials that differ by a constant factor as identical. 
Therefore, we can think of the space of polynomials as the projective space $\PP\RR^{\binom{k+2}{2}}$.

A family of curves in $\RR^2$ is $s$\emph{-dimensional} if it corresponds to a variety $\family \subseteq \PP\RR^{\binom{D+2}{2}}$ of a constant complexity and dimension $s$.
For example, the set of circles in $\RR^2$ is a 3-dimensional family and the set of circles of radius 1 is a 2-dimensional family.
The following is a result of Sharir and Zahl \cite{SZ17}.
\begin{theorem} \label{th:SharirZahl}
Let $\pts$ be a set of $m$ points and let $\curves$ be a set of $n$ algebraic curves from an $s$-dimensional family, each of degree at most $k$. 
Assume that no two curves of $\curves$  share a common one-dimensional component.
Then for every $\eps>0$ we have
\[ I(\pts,\curves) = O_{k,s,\eps}\left(m^{\frac{2s}{5s-4}+\eps}n^{\frac{5s-6}{5s-4}}+m^{2/3}n^{2/3}+m+n\right).\]
\end{theorem}

Let $\pts$ be a set of points in $\RR^2$ and let $r\ge 2$ be an integer. 
We say that a variety $\gamma\subset \RR^2$ is $r$-\emph{rich} if it contains at least $r$ points of $\pts$. 
The following can be easily obtained from Theorem \ref{th:SharirZahl} using a standard incidence argument. 
\begin{corollary} \label{co:SharirZahlDual}
Let $\pts$ be a set of $m$ points and let $r$ be larger than some sufficiently large constant. 
Consider an $s$-dimensional family $\family$ of algebraic curves of degree at most $k$ in $\RR^2$.
For every $\eps>0$, the maximum number of $r$-rich curves in the family with no two sharing a one-dimensional component is
\[ O_{k,s,\eps}\left(\frac{m^{s+\eps}}{r^{(5s-4)/2}}+\frac{m^2}{r^3}+\frac{m}{r}\right). \]
\end{corollary}

\parag{Monotone patches.}
Let $U\subset \RR^3$ be a surface of degree $D$. We define a \emph{monotone patch} of $U$ to be a connected two-dimensional open\footnote{Whenever we refer to open sets, we mean open according to the Euclidean topology (rather than the Zariski toplogy).} semi-algebraic subset of $U_\text{reg}$ whose projection on the $xy$-plane is injective. In other words, every line parallel to the $z$-axis intersects the patch at most once.

\begin{theorem} \label{th:MonPatches}
Let $U\subset \RR^3$ be a surface of degree $D$ that does not contain lines parallel to the $z$-axis. Then $U$ can be partitioned into a variety $U_0$ of dimension at most one and a set of monotone patches $U_1,\ldots,U_M$, with the following properties. The sets $U_0,U_1,\ldots,U_M$ are pairwise disjoint and their union is $U$. The complexity of $U_0$ is $O_D(1)$ and the number of patches is $M=O_D(1)$.
\end{theorem}

Theorem \ref{th:MonPatches} is a special case of \emph{cylindrical decomposition}. Thus, the proof follows from the proof of the cylindrical decomposition theorem For example, see \cite[Section 5.1]{BPR06}. 

\parag{Crossing numbers.} The \emph{crossing number} of a graph $G=(V,E)$, denoted $\crs(G)$, is the smallest integer $k$ such that we can draw $G$ in the plane with $k$ edge crossings. 
The well-known crossing lemma \cite{ACNS82,Lei83} provides a lower bound for the crossing number of graphs with many edges.

\begin{lemma}[The crossing lemma] \label{le:Cross}
Let $G=(V,E)$ be a graph with $|E|\ge 4|V|$. Then $\crs(G) = \Omega\left(|E|^3/|V|^2\right)$.
\end{lemma}

Lemma \ref{le:Cross} assumes no parallel edges. 
The following is a crossing lemma for graphs with parallel edges (for example, see \cite{Szek97}).
The \emph{multiplicity} of an edge $e$ is the number of edges that have the same endpoints as $e$ (including $e$ itself).

\begin{lemma} \label{le:CrossMult}
Let $G = (V,E)$ be a multigraph with $|E| \ge 5\mu |V|$.
Assume that every edge has multiplicity at most $\mu$.
Then 
\[ \crs(G) = \Omega\left(\frac{|E|^3}{\mu|V|^2}\right). \]
\end{lemma}

Let $G$ be a graph and let $U$ be a monotone patch of a surface in $\RR^3$.
Since $U$ is homeomorphic to any connected open subset of the plane, the crossing number of $G$ does not change when considering drawings of $G$ on $U$. 
That is, $G$ can be drawn on $U$ with $\crs(G)$ crossings, but not with a smaller number of crossings.

\subsection{Properties of surfaces in $\RR^3$} \label{ssec:Surfaces}

We now derive several lemmas that are used in the proof of Theorem \ref{th:SurfDD}.
These lemmas study properties of surfaces in $\RR^3$. 
The reader may safely skip these lemmas, returning to them as required when going over the proof of Theorem \ref{th:SurfDD}.
Recall that we define a surface as an irreducible variety of dimension two.

\parag{Surfaces of revolution.} A surface $U\subset \RR^3$ is a \emph{surface of revolution} if there exists a line $\ell$ such that every rotation around $\ell$ is a symmetry of $U$.
We refer to $\ell$ as the \emph{axis} of $U$.
Note that $U$ is a union of disjoint circles with axis $\ell$.
Planes and spheres are surfaces of revolution with infinitely many axes.

\begin{lemma} \label{le:OneAxis}
Excluding planes and sphere, every surface of revolution has exactly one axis.
\end{lemma}
\begin{proof}[Proof sketch.]
Assume that $U$ is a surface of revolution that has distinct axes $\ell$ and $\ell'$. 
If $\ell$ and $\ell'$ are parallel then $U$ is a plane orthogonal to both lines.
We may thus assume that $\ell$ and $\ell'$ are not parallel.

Since $\ell$ is an axis of $U$, we have that $U$ is a union of circles with axis $\ell$.
As we rotate $\RR^3$ around $\ell'$, the line $\ell$ has infinitely many distinct directions.
Since $U$ remains unchanged by such rotations and $\ell$ remains an axis of $U$, there are infinitely many directions $v$ such that $U$ is the union of circles orthogonal to $v$. 
Thus, every point of $U$ is incident to infinitely many circles that are contained in $U$. 
Only planes and spheres have this property (for example, see \cite{SSZ15}).
\end{proof}

\begin{lemma} \label{le:CircsForceRevo}
For every integer $k$ there exists a constant $\delta_k = O_k(1)$ that satisfies the following. 
Let $U\subset \RR^3$ be a surface of degree $k$.
For a line $\ell\subset \RR^3$, let $\circs_\ell$ be the set of circles with axis $\ell$ that are contained in $U$. 
If $|\circs_\ell|>\delta_k$ then $U$ is a surface of revolution with axis $\ell$.
\end{lemma}
\begin{proof}
By rotating and translating $\RR^3$, we may assume that $\ell$ is the $z$-axis.
Then, every circle of $\circs_\ell$ can be defined using two parameters $h,r$ as
\[ C_{h,r} = \left\{ \left(r\cdot \frac{2t}{1+t^2}, r\cdot \frac{1-t^2}{1+t^2},h\right)\ :\ t\in \RR \right\}. \]
(One point of the circle is missing in this parametrization: $(0,-r,h)$.) 

Let $f\in \RR[x,y,z]$ be a degree $k$ polynomial satisfying $\vb(f)=U$ (after the above transformation).
Then $C_{h,r} \subset U$ if and only if 
\[ f\left(r\cdot \frac{2t}{1+t^2}, r\cdot \frac{1-t^2}{1+t^2},h\right) \]
is identically zero.
This expression is a rational function in $t$, which may not be a polynomial.
We define
\[ g(t) = f\left(r\cdot \frac{2t}{1+t^2}, r\cdot \frac{1-t^2}{1+t^2},h\right) \cdot (1+t^2)^k. \]

Note that $g(t)$ is a polynomial in $t$ of degree at most $2k$.
In addition, $C_{h,r} \subset U$ if and only if $g(t)$ is identically zero.  
That is, all the coefficients of $g(t)$ are zero.
Each such coefficient is a polynomial in $h$ and $r$.

Consider another plane $\RR^2$ with coordinates $h$ and $r$.
Let $W$ be the set of points $(h,r)\in \RR^2$ that satisfy $C_{h,r}\subset U$. By the preceding paragraph, $W$ is a variety of complexity $O_k(1)$.
Indeed, $W$ is defined by asking each of the coefficients of $g(t)$ to vanish. 
By Theorem \ref{th:comps}, there exists $\delta_k$ such that $W$ consists of at most $\delta_k$ components.
By assumption $|\circs_\ell| > \delta_k$, so $W$ is not zero-dimensional.
That is, $\circs_\ell$ is a one-dimensional family of circles with axis $\ell$.

For $0 \le \alpha < 2\pi$, let $R_\alpha: \RR^3\to \RR^3$ be a rotation of angle $\alpha$ around $\ell$.
Set $U_\alpha = U \cap R_\alpha(U)$. 
In other words, $U_\alpha$ is the set of points of $U$ that remain in $U$ after a rotation of angle $2\pi-\alpha$ around $\ell$. 
Note that $U_\alpha = \vb(f,f\circ R_{2\pi-\alpha})$, so it is a variety. 
Since $\circs_\ell$ is an infinite family of disjoint circles with axis $\ell$, we get that $U_\alpha$ is two-dimensional.
Since $U$ is irreducible and $U_\alpha\subseteq U$, we conclude that $U=U_\alpha$. 

The preceding paragraph holds for every $0 \le \alpha < 2\pi$.
That is, every rotation around $\ell$ is a symmetry of $U$. 
In other words, $U$ is a surface of revolution with axis $\ell$.
\end{proof}

Recall that we defined aligned and perpendicular circles in the introduction. 

\begin{lemma} \label{le:AlignPerpInSurf} 
Let $U\subset \RR^3$ be a surface of degree $k$. Let $\circs$ be a set of circles that are contained in $U$, such that every two circles of $\circs$ are either aligned or perpendicular.  \\[2mm]
(a) If $U$ is not a surface of revolution then $|\circs|=O_k(1)$. \\[2mm]
(b) Assume that $U\subset \RR^3$ is a surface revolution with a unique axis $\ell$.
If no circle of $\circs$ has axis $\ell$ then  $|\circs|=O_k(1)$.
\end{lemma}
\begin{proof} 
(a) For a plane $H\subset \RR^3$, let $\circs_H$ be the set of circles of $\circs$ that are contained in $H$.
Since circles in the same plane cannot be perpendicular, every two circles of $\circs_H$ are aligned.
That is, every two circles of $\circs_H$ are concentric.
Since $U$ is not a plane, $\gamma = U\cap H$ is a variety of dimension at most one and complexity $O_k(1)$.
Let $\ell$ be a line contained in $H$ and incident to the center of the circles of $\circs_H$.
Then $\ell$ intersects every circle of $\circs_H$ in two distinct points. 
By Theorem \ref{th:comps}, when $|\circs_H|$ is sufficiently large, the intersection $\ell \cap \gamma$ cannot be zero dimensional.
That is, we obtain that $\ell\subset \gamma$. 
Since this holds for every such line $\ell$, we get that $\gamma=H$.
This implies that $U=H$, contradicting the assumption that $U$ is not a surface of revolution.
We conclude that $|\circs_H|=O_k(1)$.

Let $\ell\subset \RR^3$ be a line and let $\circs_\ell$ be the set of circles of $\circs$ that have $\ell$ as their axis. 
By Lemma \ref{le:CircsForceRevo}, we have that $|\circs_\ell|=O_k(1)$.
That is, the maximum number of circles of $\circs$ that have a common axis is $O_k(1)$.

Assume that $\circs$ contains two aligned circles $C_1$ and $C_2$ that do not lie in the same plane.
No circle can be perpendicular to both $C_1$ and $C_2$, since the center of such a circle would need to lie on two parallel planes. 
Thus, all the circles in $\circs$ are aligned.
By the preceding paragraph, in this case $|\circs|=O_k(1)$.

No four planes in $\RR^3$ are pairwise perpendicular. 
This implies that no four circles in $\RR^3$ are pairwise perpendicular.
Thus, $\circs$ contains at most three pairwise perpendicular circles. 
By the above, the plane of each of these circles contains $O_k(1)$ circles of $\circs$.
We conclude that, if $\circs$ contains perpendicular circles then $|\circs|=O_k(1)$

We handled the case where $\circs$ contains two aligned circles that are not coplanar. 
We also handled the case where $\circs$ contains two perpendicular circles.
These complete the proof, since no other case remains.

(b) We claim that the proof of part (a) also holds for part (b).
In that proof, we only used $U$ not being a surface of revolution when applying Lemma \ref{le:CircsForceRevo}.
Due to the assumption that the circles of $\circs$ do not have axis $\ell$, we may apply Lemma \ref{le:CircsForceRevo} as before.
\end{proof}

\begin{lemma} \label{le:SmallAngle}
Let $U\subset \RR^3$ be a surface of degree $k$.
Let $R$ be a rotation around a line $\ell$ that is a symmetry of $U$. 
If the angle of $R$ is smaller than a constant depending only on $k$, then $U$ is a surface of revolution with axis $\ell$.
\end{lemma}
\begin{proof}
Consider a point $p\in U$ and let $H$ be the plane orthogonal to $\ell$ and incident to $p$.
Let $C$ be the circle in $H$ that is incident to $p$ and with axis $\ell$. 
By applying $R$ to $p$, we obtain additional points of $U$ that are in $C$.
Taking  the angle of $R$ to be sufficiently small, we may assume that $U \cap C$ is arbitrarily large.
Then, Theorem \ref{th:comps} implies that the intersection $U\cap C$ is not zero-dimensional.
This implies that $C\subset U$.

We may assume that $U$ is not a plane, since otherwise we are done.
Then $U$ intersects infinitely many planes orthogonal to $\ell$.
By the preceding paragraph, $U$ contains a circle with axis $\ell$ in each of those planes. 
By Lemma \ref{le:CircsForceRevo}, we conclude that $U$ is a surface of revolution with axis $\ell$.
\end{proof}

\parag{Ruled surfaces.} 
Let $U$ be a surface in $\RR^3$. 
As stated in the introduction, $U$ is \emph{ruled} if for every point $p\in U$ there exists a line that is contained in $U$ and incident to $p$.
We only require one basic property of non-ruled surfaces.
For more details, see for example \cite[Corollary 3.3]{GK15}.

\begin{lemma} \label{le:Ruled} 
A non-ruled surface of degree $D$ in $\RR^3$ contains $O_D(1)$ lines. 
\end{lemma}

\parag{Cylindrical surfaces.} A surface $U\subset \RR^3$ is a \emph{cylindrical surface} if there exists a curve $\gamma$ that satisfies the following. 
The curve $\gamma$ is contained in a plane $H\subset \RR^3$ with normal $\vec{v}$.
The surface $U$ is the union of the lines in $\RR^3$ of direction $\vec{v}$ that are incident to a point of $\gamma$.
Equivalently, $U$ is the union of the translations of $\gamma$ in direction $\vec{v}$.
Note that a plane is a cylindrical surface and that every cylindrical surface is also a ruled surface. 

\begin{lemma} \label{le:parallelRef}
Let $U\subset \RR^3$ be a surface with reflectional symmetries about the distinct planes $H_1$ and $H_2$.
If $H_1$ and $H_2$ are parallel then $U$ is a cylindrical surface.
\end{lemma}
\begin{proof}
Let $\delta$ denote the distance between $H_1$ and $H_2$.
Let $\vec{v}$ be a vector in the direction orthogonal to $H_1$ and $H_2$.
It is not difficult to verify that composing the two reflections leads to a translation $T$ of distance $\delta$ in direction $\pm\vec{v}$.
By definition, $T$ is also a symmetry of $U$. 

Let $p$ be a point of $U$.
Let $\ell\subset \RR^3$ be the line incident to $p$ and of direction $\vec{v}$.
By repeatedly applying $T$ on $p$, we obtain that $\ell$ contains infinitely many points of $U$.
By Theorem \ref{th:comps}, the intersection $U\cap \ell$ is not zero-dimensional.
This implies that $\ell\subseteq U$.
In other words, for every $p\in U$ there exists a line $\ell\subset U$ incident to $p$ and of direction $\vec{v}$.
We conclude that $U$ is a cylindrical surface.
\end{proof}

\subsection{Proof of Theorem \ref{th:SurfDD}} \label{ssec:SurfDDproof}

We are now ready to prove our main result. We first recall the statement of this theorem.
\vspace{2mm}

\noindent {\bf Theorem \ref{th:SurfDD}.}
\emph{Let $U\subset \RR^3$ be a non-ruled surface of degree $k$.
Let $\pts$ be a set of $n$ points on $U$.
Then for any $\eps>0$, we have 
 }
\[ D(\pts) = \Omega_k\left(n^{32/39-\eps}\right). \]

\begin{proof}
Since $U$ is non-ruled, it is not a plane. By Theorem \ref{th:Tao}, we may also assume that $U$ is not a sphere. (We may also assume that $U$ is not a hyperboloid of two sheet, but this is not necessary for our proof.)

By Lemma \ref{le:Ruled}, the surface $U$ contains $O_k(1)$ lines. 
If one of these lines contains $\Omega(n/\log n)$ points of $\pts$, then these points span $\Omega(n/\log n)$ distinct distances. 
We may thus assume that each of these lines contains $O(n/\log n)$ points of $\pts$.
We discard from $\pts$ all points that lie on a line contained in $U$.
After a generic rotation around the origin, we may also assume that $U$ does not contain lines that are parallel to the $z$-axis.

We apply Theorem \ref{th:MonPatches}, to obtain a variety $U_0$ and monotone patches $U_1,\ldots,U_M$. 
By definition, $U_0$ is of dimension at most one and complexity $O_k(1)$.
We may assume that $U_0$ contains $O(n/\log n)$ points of $\pts$.
Otherwise, Lemma \ref{le:DDcurve} implies a result stronger than required.
Since $M=O_k(1)$, by the pigeonhole principle there exists a monotone patch $U_j$ that contains $\Theta_k(n)$ points of $\pts$.
We discard from $\pts$ the points that are not on $U_j$.
Abusing notation, we refer to the set of remaining points as $\pts$ and to $U_j$ as $U$.

\parag{Constructing a multigraph.}
Recall that, for a point $p\in\RR^3$ and $r>0$, we denote by $S(p,r)$ the sphere of radius $r$ centered at $p$.
We define the \emph{level set} 
\[ L(p,r)=U \cap S(p,r). \]
That is, $L(p,r)$ is the set of points on $U$ at distance $r$ from $p$.
Since $U$ is not a patch of a sphere, every level set is a semi-algebraic set of dimension at most one and complexity $O_k(1)$. 

Let $\lsets$ be the set of level sets $L(p,r)$ that are incident to at least one point of $\pts$ and with $p\in\pts$. 
For every level set $L(p,r)$ from $\lsets$, we discard the zero-dimensional components and the singular points of $L(p,r)$.
By Lemma \ref{le:singular}, this removes $O_k(1)$ points from each level set. 
We discard empty sets from $\lsets$.
Each remaining set is a semi-algebraic set of complexity $O_k(1)$.
We partition every remaining element of $\lsets$ to its connected components.
By Lemma \ref{le:ConnSA}, each such set consists of $O_k(1)$ connected components. 

After the above process, $\lsets$ consists of segments of curves with no singular points. 
Some of those may also be closed.
Since the level sets are contained in spheres, all segments are bounded.
We discard from $\lsets$ segments that contain at most one point of $\pts$.

Following Sz\'ekely's technique \cite{Szek97} for the planar case, we construct a graph $G=(V,E)$ as follows.
We place a vertex in $V$ for each point of $\pts$.
We go over the segments of $\lsets$ and connect every two vertices that are consecutive along such a segment. 
By ``consecutive'' we mean that, when traveling along the segment, there are no other points of $\pts$ between the two vertices.
Since the segments do not contain singular points, no segment intersects itself.
This means that the process of travelling along a segment is well-defined.
For an example of how edges are added, see Figure \ref{fi:GraphEdges}.
Since segments from the same level set are disjoint, such segments cannot lead to parallel edges. 
Thus, an edge $(u,v)$ has multiplicity $r$ when $u$ and $v$ are consecutive along segments from level sets of $r$ distinct points of $\pts$.

\begin{figure}[ht]
    \centering
    \begin{subfigure}[b]{0.3\textwidth}
    \centering
        \begin{tikzpicture}[scale=0.7, every node/.style={fill = blue, draw=black,circle, minimum size = 0.5mm}]
        \draw[domain=20:170,samples=500] plot (\x:{4 + 0.2 * sin (5 * \x)});
        \draw[domain=0:180,samples=500, dashed] plot (\x:{2.6 + 2 * cos (2 * (\x + 90))});
        \node[label=$z$] (N-z) at (30:{4 + 0.2 * sin (5 * 30)}) {};
        \node[label=$y$] (N-y) at (71:{4 + 0.2 * sin (5 * 71)}) {};
        \node[label=$x$] (N-x) at (115:{4 + 0.2 * sin (5 * 115)}) {};
        \node[label=$w$] (N-w) at (160:{4 + 0.2 * sin (5 * 160)}) {};
        \node[label=$v$] (N-v) at (30:{2.6 + 2 * cos (2 * 120)}) {};
        \node[label=$u$] (N-u) at (150:{2.6 + 2 * cos (2 * 240}) {};
    \end{tikzpicture}
    \caption{Example of two level sets and the points on them.}
    \end{subfigure}
    \hspace{2cm}
    \begin{subfigure}[b]{0.3\textwidth}
    \centering
        \begin{tikzpicture}[scale=0.7, every node/.style={fill = blue, draw=black,circle, minimum size = 0.5mm}]
            \node[label=$z$] (N-z) at (30:{4 + 0.2 * sin (5 * 30)}) {};
            \node[label=$y$] (N-y) at (71:{4 + 0.2 * sin (5 * 71)}) {};
            \node[label=$x$] (N-x) at (115:{4 + 0.2 * sin (5 * 115)}) {};
            \node[label=$w$] (N-w) at (160:{4 + 0.2 * sin (5 * 160)}) {};
            \node[label=$v$, right] (N-v) at (30:2) {};
            \node[label=$u$, left] (N-u) at (150:2.4) {};
        
            \path (N-w) edge[-] (N-x);
            \path (N-x) edge[-, bend left = 50] (N-y);
            \path (N-y) edge[-] (N-z);
            \path (N-u) edge[dashed] (N-x);
            \path (N-x) edge[dashed, bend right = 50] (N-y);
            \path (N-y) edge[dashed] (N-v);
        \end{tikzpicture}
        \caption{Graph edges corresponding to the two level sets.}
    \end{subfigure}
    \caption{Example of edges corresponding to two level sets. There are two edges between vertices $x$ and $y$ since they are consecutive along both level sets.} \label{fi:GraphEdges}
\end{figure}
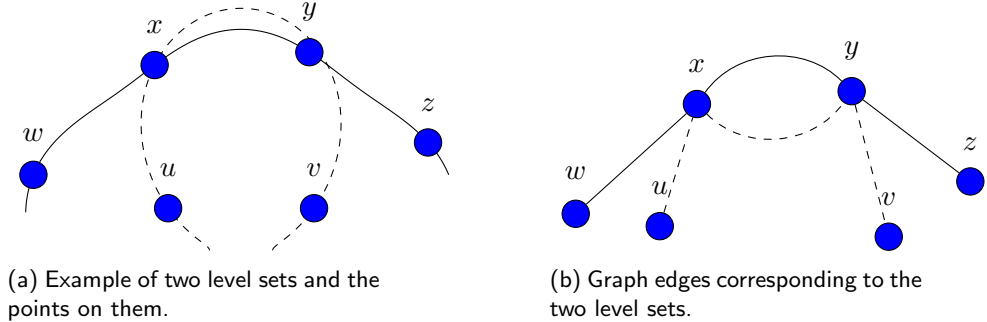

We claim that the number of edges in $G$ is $\Theta(n^2)$. 
In particular, for each $p\in \pts$, the level sets of $p$ contribute $\Theta(n)$ edges to $G$.
Indeed, fix $p\in \pts$. 
The level sets of $p$ contain all points of $\pts$.
We may assume that $D(\{p\},\pts\setminus\{p\})=O(n/\log n)$, since otherwise we are done. 
In other words, the number of level sets $L(p,r)$ that contain at least one point of $\pts$ is $O(n/\log n)$. 
By Theorem \ref{th:comps}, when discarding zero-dimensional components from the level sets, we lose $O_k(n/\log n)$ points of $\pts$.  
For the same reasons, the level sets of $p$ contribute $O_k(n/\log n)$ segments to $\lsets$. These segments are disjoint, since they do not contain singular points. 

By the above, the level sets of a fixed $p\in \pts$ correspond to  $O(n/\log n)$ disjoint segments of $\lsets$ that contain $\Theta(n)$ points of $\pts$. A closed segment that contains $r$ points of $\pts$ contributes $r$ edges to $E$. An open segment that contains $r$ points of $\pts$ contributes $r-1$ edges. 
We conclude that the level sets of $p$ contribute $\Theta(n)$ edges to $E$, which in turn implies that $|E|=\Theta(n^2)$.

\parag{The crossing number.}
The proof of the theorem is based on double counting $\crs(G)$.
Recall that crossing numbers do not change when switching between the plane and a monotone patch.
First, we draw $G$ on $U$ according to its geometric representation.
That is, every vertex of $V$ is placed at the corresponding point of $\pts$.
Every edge of $E$ is drawn as the arc of the level set that led to it. These edges are disjoint from the boundary of $U$.

The intersection of two level sets $L(p_1,r_1)$ and $L(p_2,r_2)$ is $U \cap S(p_1,r_1) \cap S(p_2,r_2)$.
Set $C = S(p_1,r_1) \cap S(p_2,r_2)$ and note that $C$ is either a circle, a single point, or an empty set.
By Theorem \ref{th:comps}, either $C \subset U$ or $|C\cap U| = O_k(1)$.
That is, the intersection of two level sets is either empty, a set of $O_k(1)$ points, or a circle.

When two level sets intersect at a circle $C$, this may lead to two parallel edges of $E$ that are drawn as the same arc of $C$.
In this case, we slightly move the drawing of one of the edges, so that the two edges no longer intersect in their interiors. 
We obtain that the interiors of every two edges of $E$ intersect in $O_k(1)$ points. 

Set $t=D(\pts)$.
For every $p\in \pts$, the number of level sets of $p$ that contain at least one point of $\pts$ is at most $t$.
Thus, there are $O(nt)$ level sets that contain at least one point of $\pts$.
By the above, the number of intersections between the edges that originate from two different level sets is $O_k(1)$. Slightly moving identical edges may increase the number of intersections, but not by more than a constant factor.
We conclude that 
\begin{equation}\label{eq:CrossingUpper}
\crs(G) = O_k(n^2t^2).
\end{equation}

We would like to apply Lemma \ref{le:CrossMult} to obtain a lower bound for $\crs(G)$. 
However, the multiplicity of some edges might be high, which  makes the lemma inefficient. 
To address this issue, we rely on the following lemma.
Recall that the \emph{perpendicular bisector} of two points $p,q\in\RR^3$ is the set of all points $a\in \RR^3$ that satisfy $|pa|=|qa|$.
Equivalently, it is the plane incident to the midpoint of $p$ and $q$ and orthogonal to the vector $\overrightarrow{p-q}$.

\begin{lemma} \label{le:RichBisectors}
For an integer $r\ge2$, let $T_r$ be the set of pairs $(H, e)$ such that $e = (u, v)\in E$, the plane $H\subset \RR^3$ is the perpendicular bisector of $u$ and $v$, and $H$ is incident to at least $r$ points of $\pts$.
If $e$ and $e'$ are parallel edges, then $(H, e)$ and $(H, e')$ represent two distinct pairs in $T_r$. 
Then,
\[ |T_r| = O_k\left(\frac{n^{3+\eps}t^{3/4}}{r^{9/2}}+ \frac{n^2t^{3/4}}{r^{2}} + nt\log n\right). \]
\end{lemma}

If vertices $u,v\in V$ have more than $r$ edges between them, then they are consecutive
on more than $r$ level sets. This in turn implies that the perpendicular bisector of $u$ and $v$ is
$r$-rich. 
We remove from $G$ all edges with multiplicity at least $r=\Theta(n^{2/9+\eps}t^{1/6})$.
By Lemma \ref{le:RichBisectors} and recalling that $t=O(n/\log n)$, we get that $|T_r|=O_k(n^{2-\eps})$.
Thus, we still have $|E|=\Theta(n^2)$.

After the above edge removal, every edge in $E$ has multiplicity $O(n^{2/9+\eps}t^{1/6})$.
Lemma \ref{le:CrossMult} implies 
\begin{equation*}
\crs(G) = \Omega_k\left(\frac{|E|^3}{n^{2/9+\eps}t^{1/6}|V|^2}\right) = \Omega_k\left(\frac{n^{34/9-\eps}}{t^{1/6}}\right).
\end{equation*}

Combining this bound with \eqref{eq:CrossingUpper} yields $t=\Omega_k(n^{32/39-\eps})$, as required.
To complete the proof, it remains to prove Lemma \ref{le:RichBisectors}.
We do so in the following section. 

\subsection{Proof of Lemma \ref{le:RichBisectors}} \label{ssec:RichEdgesLemma}

\parag{Rich bisectors.}
By performing a generic rotation of $\RR^3$ around the origin at the beginning of the proof of Theorem \ref{th:SurfDD}, we may assume the following.
For every pair of distinct points $(p,q)\in\pts^2$, the normal to the perpendicular bisector of $p$ and $q$ does not have a zero $z$-coordinate.
In other words, every such bisector can be written as $ax+by+z=c$ for some $a,b,c\in \RR$.

Consider the family of planes in $\RR^3$ defined as $ax+by+z=c$ with parameters $a,b,c\in \RR$.
Taking the intersection of every such plane with $U$ leads to a family $\family$ of varieties of dimension at most one and complexity $O_k(1)$.
An element in this family is defined by the three parameters $a,b,c$.
Since every two planes are either disjoint or intersect at a line, every two elements of $\family$ intersect either in $O_k(1)$ collinear points or in a line. 

For a line $\ell$ contained in the original surface $U$, let $\family_\ell$ be the set of varieties of $\family$ that contain $\ell$. 
Since these varieties are intersections of planes with $U$, the intersection of every two varieties of $\family_\ell$ is $\ell$. 
Recalling that no point of $\pts$ is on $\ell$, we observe that every point of $\pts$ is incident to at most one element of $\family_\ell$.
In other words, $I(\pts,\family_\ell) \le n$.

Let $\family_\text{line}$ be the set of elements of $\family$ that contain at least one line that is contained in $U$. 
Let $\family'$ be the set of elements of $\family$ that do not contain any line that is in $U$. 
That is, $\family' = \family \setminus \family_\text{line}$.
By Lemma \ref{le:Ruled}, the original surface $U$ contains $O_k(1)$ lines.
By combining this with the preceding paragraph, we obtain that $I(\pts,\family_\text{line})=O_k(n)$.
This implies that the number of $r'$-rich elements of $\family_\text{line}$ is $O_k(n/r')$.

We now study $I(\pts,\family')$. 
Since the elements of $\family'$ contain no lines, every two intersect in $O_k(1)$ points. 
We consider a projection of $\pts$ and $\family'$ onto a generic plane $\Pi$.
Let $\pts^*$ be the set of projections of the points of $\pts$ on $\Pi$.
Let $\family^*$ be the set of Zariski closures of the projections of the elements of $\family'$.
By Lemma \ref{le:projection}, the set $\family^*$ consists of varieties of dimension at most one and complexity $O_k(1)$.
Since $\Pi$ is chosen generically, we may assume that no two points of $\pts$ are projected into the same point, that every two elements of $\family'$ intersect in $O_k(1)$ points, and that $I(\pts,\family') = I(\pts^*,\family^*)$.

Recall that $\family$ is contained in a 3-parameter family.
This means that $\family^*$ is also contained in a 3-parameter family.
By thinking of $\Pi$ as $\RR^2$, we may apply Corollary \ref{co:SharirZahlDual} with $\family^*$.
By applying the corollary with $\eps$, $s=3$, and $\Theta(n)$ points, the number of $r'$-rich elements of $\family'$ is  
\begin{equation}\label{eq:Krich}
O_k\left(\frac{n^{3+\eps}}{(r')^{11/2}}+\frac{n^2}{(r')^3}+\frac{n}{r'}\right).
\end{equation}
 
Since the number of $r$-rich elements of $\family_\text{line}$ is $O(n/r')$, we get that
\eqref{eq:Krich} is also a bound for the number of $r'$-rich elements in $\family$.

\parag{Bounding $|T_r|$.}
Let $H$ be a plane incident to $r'$ points of $\pts$. 
Each of these $r'$ points defines at most $t$ level sets that contain at least one point of $\pts$.
By Theorem \ref{th:comps}, such a level set has $O_k(1)$ intersection points with $H$ (ignoring circles fully contained in $H$).
Since each such intersection point leads to at most one pair $(H,e)\in T_r$, we get that $H$ participates in $O_k(r't)$ such pairs. 

In the following, all logarithms have base 2.
For $j>\log\log n$, we consider planes that are $2^j$-rich with respect to $\pts$ but not $2^{j+1}$-rich. 
We refer to such a plane as $2^j$-\emph{fixed}. Below we prove that, excluding $O_k(n/2^j)$ planes, every $2^j$-fixed plane $H$ forms a pair $(H,e)\in T_r$ with $O_k(2^jt^{3/4})$ edges $e\in E$. We refer to the excluded planes as \emph{exceptional planes}.
We first assume that this claim holds and derive the assertion of the lemma.
Afterwards, we prove this claim.

For a positive integer $j$, let $\planes_j$ be the set of planes that are $2^j$-fixed. 
Setting $r'=2^j$ in \eqref{eq:Krich} leads to an upper bound for $|\planes_j|$. 
Combining the above and assuming that $r> \log n$ gives
\begin{align*}
|T_r| &\le \sum_{j=\log r}^{\log n} \left(|\planes_j|\cdot O\left(2^j\cdot t^{3/4}\right)+ O_k\left(\frac{n}{2^j}\cdot 2^jt\right)\right) = \sum_{j=\log r}^{\log n} O_k\left(\frac{n^{3+\eps}t^{3/4}}{2^{9j/2}}+\frac{n^2t^{3/4}}{2^{2j}}+nt\right) \\[2mm]
&\hspace{55mm} = \sum_{j=\log r}^{0.5\log n} O_k\left(\frac{n^{3+\eps}t^{3/4}}{2^{9j/2}}+\frac{n^2t^{3/4}}{2^{2j}}\right)+ \sum_{j=0.5\log n}^{\log n}\hspace{-3mm} O_k\left(nt\right) \\[2mm]
&\hspace{55mm} = O_k\left(\frac{n^{3+\eps}t^{3/4}}{r^{9/2}}+\frac{n^2t^{3/4}}{r^{2}} + nt\log n\right).
\end{align*}

It remains to prove the above claim about the non-exceptional planes.

\parag{Non-exceptional planes.}
Let us recall the details of what we need to prove. 
For $j>\log\log n$, we need to show that, excluding $O_k(n/2^j)$ exceptional planes, every $2^j$-fixed plane $H$ forms a pair $(H,e)\in T_r$ with $O_k(2^jt^{3/4})$ edges $e\in E$. 

Let $H$ be a $2^j$-fixed plane. 
For a point $p\in \RR^3$, let $H(p)$ be the reflection of $p$ about $H$. 
Let $\pts_H$ be the set of points participating in an edge that forms a pair with $H$. Note that $\pts_H$ is contained in $U \cap H(U)$.
We define the \emph{type} of $H$  according to the reflection that it induces.
\begin{itemize}
\item The plane $H$ is \emph{Type 1} if $U=H(U)$ (that is, if $H(\cdot)$ is a symmetry of $U$).
\item The plane $H$ is \emph{Type 2} if $U\cap H(U)$ contains $O_k(t^{3/4})$ points of $\pts$.
\item The plane $H$ is \emph{Type 3} if a circle in $U\cap H(U)$ contains $\Omega_k(t^{3/4})$ points of $\pts$. 
\item The plane $H$ is \emph{Type 4} if a non-circle component of $U\cap H(U)$ contains $\Omega_k(t^{3/4})$ points of $\pts$. 
\end{itemize}

Since $U\cap H(U)$ consists of  $O_k(1)$ components, $H$ must have one of the four types.
If $H$ satisfies the conditions of more than one case, then the type of $H$ is the first case satisfied. 
For example, if the reflection $H(\cdot)$ is a symmetry of $U$ and $U\cap H(U)$ contains a circle with $\Omega_k(t^{3/4})$ points of $\pts$, then $H$ is Type 1 rather than Type 3. We study the number of planes of each type separately.

\parag{Type 1 planes.}
We begin with the special case where $U$ is a surface of revolution.
Since $U$ is neither a plane nor a sphere, Lemma \ref{le:OneAxis} implies that $U$ has a unique axis $\ell$.
Any Type 1 plane is either orthogonal to $\ell$ or contains $\ell$. 
Since $U$ is not ruled, it is not a cylindrical surface.
By Lemma \ref{le:parallelRef}, at most one Type 1 plane is orthogonal to $\ell$.
The intersection of any two planes that contain $\ell$ is $\ell$. 
Since $|\ell\cap \pts| = O_k(1)$ and $j>\log\log n$, the number of  $2^j$-fixed planes that contain $\ell$ is $O(n/2^j)$. 
We conclude that the number of exceptional planes of Type 1 is $O(n/2^j)$.

Next, consider the case where $U$ is not a surface of revolution.
Let $H_1$ and $H_2$ be distinct planes that correspond to reflectional symmetries of $U$. 
By Lemma \ref{le:parallelRef}, the planes $H_1$ and $H_2$ are not parallel. 
Assume that the angle between $H_1$ and $H_2$ is smaller than some sufficiently small constant $\alpha$.
Then $H_1\circ H_2$ is a rotation  of sufficiently small angle. 
Lemma \ref{le:SmallAngle} implies that $U$ is a surface of revolution.
This contradiction implies that the angle between $H_1$ and $H_2$ is at least $\alpha$.

Let $\planes$ be a set of planes in $\RR^3$, such that no two are parallel and no two form an angle smaller than $\alpha$. 
We claim that $|\planes|=O_k(1)$.
Indeed, let $S\subset \RR^3$ be the unit sphere centered at the origin. 
For every plane $H\in \planes$, we place a point $p_H$ on $S$  so that the vector from the origin to $p_H$ has the same direction as the normal of $H$.
By the above, there exists a minimum distance $\alpha$ between any two points on $S$.
In other words, when placing an open surface patch of radius $\alpha/2$ centered at each point, no two  patches intersect. 
Since the surface area of $S$ is constant and the surface area of each patch is constant, there could be $O_k(1)$ such patches.
We conclude that $|\planes|=O_k(1)$.
That is, in this case the number of Type 1 planes is $O_k(1)$. 

\parag{Type 2 planes.} Let $e=(u,v)$ be an edge satisfying $(H,e)\in T_r$.
Then $H(u)=v$ and $H(v)=u$.
Since each point on $H$ leads to at most one level set containing both $u$ and $v$, there are $O(2^j)$ edges parallel to $e$. 
We conclude that the number of edges that form a pair with $H$ in $T_r$ is $O(|\pts_H|2^j)$.
In particular, if $H$ is Type 2 then $H$ participates in $O_k(2^jt^{3/4})$ pairs $(H, e) \in T_r$. That is, Type 2 planes are not exceptional.

\parag{Type 3 planes.}
By Theorem \ref{th:CircDD}, if there exist two $\Omega_k(t^{3/4})$-rich circles that are neither perpendicular nor aligned, then $D(\pts)>t$. 
Since this is a contradiction to the definition of $t$, we may assume that no  two such circles exist.

We simultaneously handle the case where $U$ is a surface of revolution and the case where it is not. 
In the former case, we first ignore circles that have the same axis as $U$. 
By Lemma \ref{le:AlignPerpInSurf}, the number of other circles that are $\Omega_k(t^{3/4})$-rich is $O_k(1)$.
We refer to those as \emph{exceptional circles}. 
After concluding our study of the exceptional circles we return to the ignored circles.

For every pair of distinct exceptional circles $(C_1,C_2)$, at most one reflection of $\RR^3$ takes $C_1$ to $C_2$.
Thus, distinct pairs of exceptional circles lead to $O_k(1)$ exceptional planes. However, there may also exist exceptional planes that take an exceptional circle to itself.

Consider an exceptional circle $C$ and let $\ell$ be the axis of $C$.
Then all planes that takes $C$ to itself contain $\ell$ (we may ignore the plane that contains $C$).
Since $|\ell\cap \pts| = O_k(1)$ and $j>\log\log n$, the number of  $2^j$-fixed planes that contain $\ell$ is $O(n/2^j)$. 
By summing this over all $O_k(1)$ exceptional circles, we obtain $O_k(n/2^j)$ exceptional planes.

We now consider the circles that were ignored above. 
To distinguish such circles from the above exceptional circles, we refer to them as \emph{ignored circles}.
That is, we are in the case where $U$ is a surface of revolution with axis $\ell$ and consider $\Omega(t^{3/4})$-rich circles with axis $\ell$.

Every plane that maps an ignored circle to itself also contains $\ell$.
By repeating the above argument, we get $O_k(n/2^j)$ such $2^j$-fixed planes. 
Every plane that maps one ignored circle to another is orthogonal to $\ell$.
Since such planes are disjoint, there are $O(n/2^j)$ such $2^j$-fixed planes.

It remains to consider planes that map an ignored circle to an exceptional circle. 
By Theorem \ref{th:CircDD}, we may assume that all exceptional circles are perpendicular to all ignored circles. 
If there is more than one ignored circle, then no circle is perpendicular to both.
If there is a single ignored circle, then $O_k(1)$ planes take it to exceptional circles. 

\parag{Type 4 planes.} To analyse this case, we rely on the following special case of a theorem by Raz \cite{Raz20}.

\begin{theorem} \label{th:DDcurveR3}
Let $\gamma\subset \RR^3$ be an irreducible variety of dimension one and constant complexity that is neither a line nor a circle.
Let $\pts\subset \gamma$ be a set of $n$ points.
Then $D(\pts)=\Omega(n^{4/3})$.
\end{theorem}
Assume that there exists a plane $H$ of Type 4. By Theorem \ref{th:DDcurveR3}, the points on $U\cap H(U)$ determine $\Omega_k(t)$ distinct distances. 
Taking a sufficiently large constant in the definition of Type 4, we obtain that the number of distances is larger than $t$. 
This contradiction implies that no plane is of Type 4.  

By considering all the above cases, we conclude that the total number of exceptional planes is $O_k(n/2^j)$. 
This completes the proof of Lemma \ref{le:RichBisectors}, which in turn completes the proof of Theorem \ref{th:SurfDD}.
\end{proof}

\section{Few distinct distances between two circles} \label{sec:ProofA}

We now present constructions that have few distinct distances between two circles. These constructions are based on the notions of aligned and perpendicular circles, as defined in the introduction.
\vspace{2mm}

\noindent {\bf Theorem \ref{th:CircDD}(a).} 
\emph{Let $\crc_1$ and $\crc_2$ be two circles in $\RR^3$ that are either aligned or perpendicular.
Then there exist a set $\pts_1\subset \crc_1$ of $m$ points and a set $\pts_2\subset \crc_2$ of $n$ points, such that}
\[ D(\pts_1,\pts_2) = \Theta(m+n). \]

\begin{proof}
Without loss of generality, assume that $n\ge m$. Since rotations, translations, and uniform scalings of $\RR^3$ do not affect the size of $D(\pts_1,\pts_2)$, we may assume that
$\crc_1$ is the unit circle centered at the origin and contained in the $xy$-plane.
We parametrize $\crc_1$ as
\begin{equation} \label{eq:FirstCirc}
 \left\{ (t,\pm \sqrt{1-t^2},0)\ :\ -1\le t\le 1 \right\}.
 \end{equation}

\parag{The aligned case. }
We first consider the case where $\crc_1$ and $\crc_2$ are aligned.
In this case, after the above transformations, $\crc_2$ is contained in a plane of the form $z=c$ and centered at $(0, 0, c)$.
If $z=0$ then, to be aligned, the circles must be concentric. 
We may then use the construction presented in the introduction, obtaining $D(\pts_1,\pts_2)= \Theta(m+n)$. 
For $z\neq 0$, when projecting $\crc_2$ onto the $xy$-plane, we get two concentric circles.
By Lemma \ref{le:parallelPlanes}, we may again use the construction from the introduction (with the appropriate $z$-coordinates). 

\parag{The perpendicular case. }
We now consider the case where $\crc_1$ and $\crc_2$ are perpendicular.
In this case, the plane containing $\crc_2$ is incident to the origin and the center of $\crc_2$ is incident to the $xy$-plane.
By rotating around the $z$-axis, we may assume that $\crc_2$ is contained in the $xz$-plane.
Note that the center of $\crc_2$ is on the $x$-axis, and denote this center as $(-a,0,0)$.
Let $r$ denote the radius of $\crc_2$ (after the above scaling).
We parametrize $\crc_2$ as
\begin{equation} \label{eq:SecondCirc}
\left\{ \left(rs-a,0,\pm r\sqrt{1-s^2}\right) \ :\  -1\le s \le 1 \right\}.
\end{equation}

We arbitrarily choose $b\in (a,a+\min\{1,r\})$ and $\beta\in ((a/b)^{1/n},1)$.
We then consider the point sets
\begin{align*}
\pts_1 &= \left\{ \left(b\cdot\beta^j -a,\sqrt{1-(b\cdot\beta^j -a)^2},0\right)\ :\ 0\le j \le m-1\right\}, \\[2mm]
\pts_2 &= \left\{ \left(-b\cdot\beta^k,0,r\sqrt{1-\left(\frac{a-b\cdot\beta^k}{r}\right)^2}\right)\ :\ 0\le k \le n-1\right\}.
\end{align*}

To see that the points of $\pts_1$ lie on $\crc_1$, set $t=b\cdot\beta^j -a$ in \eqref{eq:FirstCirc}.
Since $n\ge m$, we have that $\beta\in ((a/b)^{1/m},1)$, which implies that $t>0$.
Since $b< a+1$, we get that $t< 1$.
Similarly, to see that the points of $\pts_2$ lie on $\crc_2$, set $s=\frac{a-b\cdot\beta^k}{r}$ in \eqref{eq:SecondCirc}.
Since $\beta\in ((a/b)^{1/n},1)$, we have that $s< 0$.
Since $b<a+r$, we have that $s>-1$.

The square of the distance between a point of $\pts_1$ and a point of $\pts_2$ is
\begin{align}
( b\cdot\beta^j -a+b\cdot\beta^k)^2 + &\left(1-(b\cdot\beta^j -a)^2\right) + \left( r^2-r^2\left(\frac{a-b\cdot\beta^k}{r}\right)^2\right) \nonumber\\
&= b^2\beta^{2j}+a^2+b^2\beta^{2k}-2ab\beta^j-2ab\beta^k+2b^2\beta^{j+k} \nonumber \\
&\hspace{22mm} 
+ 1-b^2\beta^{2j}-a^2+2b\beta^ja +  r^2-a^2-b^2\beta^{2k}+2ab\beta^k \nonumber\\
&=2b^2\beta^{j+k}  + 1 +  r^2-a^2. \label{eq:MultForm}
\end{align}

The only part of the above expression that depends on the choice of $j$ and $k$ is $\beta^{j+k}$.
Thus, there are exactly $m+n-1$ distinct differences in $\pts_1\times\pts_2$.
This completes the proof of part (a) of Theorem \ref{th:CircDD}.
\end{proof}

\section{Many distances between two circles}\label{sec:ProofB}

We now study distinct distances between circles that are neither aligned nor perpendicular. 
\vspace{2mm}

\noindent {\bf Theorem \ref{th:CircDD}(b).}
\emph{Let $\crc_1$ and $\crc_2$ be two circles in $\RR^3$ that are neither aligned nor perpendicular.
Let $\pts_1\subset \crc_1$ be a set of $m$ points and let $\pts_2\subset \crc_2$ be a set of $n$ points.
Then}
\[ D(\pts_1,\pts_2) = \Omega\left(\min\left\{m^{2/3}n^{2/3},m^2,n^2\right\}\right). \]
\begin{proof}
Since rotations, translations, and uniform scalings of $\RR^3$ do not affect $D(\pts_1,\pts_2)$, we may assume that
$\crc_1$ is the unit circle centered at the origin and contained in the $xy$-plane. Since a rotation around the $z$-axis takes $\crc_1$ to itself, we may further assume that the $y$-coordinate of the center of $\crc_2$ is 0.

We now assume that the plane containing $\crc_2$ is not parallel to the $xy$-plane and to the $xz$-plane. 
Throughout the proof we ignore several other special cases.
These special cases are handled after the general case, in Appendix \ref{app:analysis}.

\parag{Parametrizing the circles.} 
We parametrize $\crc_1$ as
\begin{equation}\label{eq:FirstCircParam}
\gamma_1(s) = \left\{\left(\frac{2s}{1+s^2}, \frac{1-s^2}{1+s^2},0\right)\ :\ s \in \RR \right\}.
\end{equation}
This parametrization overlooks one point of $\crc_1$, corresponding to the case where $s\to \infty$.
If this missing point is in $\pts_1$, then we remove it from $\pts_1$.

Let $c=(p,0,q)$ be the center of $\crc_2$ and let $r$ be the radius of $\crc_2$.
Let $H$ be the plane containing $\crc_2$ and let $H_0$ be the translation of $H$ that contains the origin.
For any two orthogonal unit vectors $V_1$ and $V_2$ that span $H_0$ (in other words, orthogonal vectors that span the set of directions of $H$), we can parametrize $\crc_2$ as
\[ \{ c+r\cdot V_1 \cdot  \cos \theta + r\cdot V_2 \cdot \sin\theta\ :\ \theta\in [0,2\pi) \}. \]

Let $V'_1$ be a unit vector contained in the intersection of $H_0$ and the $xz$-plane.
Similarly, let $V'_2$ be a unit vector contained in the intersection of $H_0$ and the $xy$-plane.
We assume that $H$ does not contain lines parallel to the $x$-axis.
This implies that $V'_1\neq V'_2$.
There exist $0 < \alpha,\beta < \pi$ such that
\begin{equation} \label{eq:V1V2} 
V'_1 = \left(\cos \alpha, 0, \sin \alpha\right) \quad \text{ and } \quad V'_2 = \left(\cos \beta, \sin \beta, 0 \right) 
\end{equation}

Note that $\beta$ is the angle between the $x$-axis and the line $H_0\cap \vb(z)$ (both lines are in the $xy$-plane).
Similarly, $\alpha$ is the angle between the $x$-axis and the line $H_0\cap \vb(y)$. 
The unit vectors $V'_1$ and $V'_2$ span $H_0$, but they might not be orthogonal.
The following replaces $V'_2$ with a unit vector in $H_0$ that is orthogonal to $V'_1$.
First consider 
\[ \vec{n} = V'_1 \times V'_2 = \left(-\sin \alpha \cdot \sin \beta,\ \sin \alpha \cdot \cos \beta,\ \cos \alpha \cdot \sin \beta\right). \]
Note that $\vec{n}$ is orthogonal to $H_0$.

We next consider
\[ V''_2 = \vec{n} \times V'_1 = \left(\sin^2 \alpha \cdot \cos \beta,\ \sin \beta,\ -\sin \alpha \cdot \cos \alpha \cdot \cos \beta\right). \]
We rename $V'_1$ and $V''_2$ as $V_1$ and $V_2$, respectively.
While $V_2$ is contained in $H_0$ and orthogonal to $V_1$, it might not be a unit vector.
To normalize $V_2$, we find
\begin{align*}
\|V_2 \| = \big(\sin^4 \alpha \cdot \cos^2 \beta + \sin^2 \beta &+ \sin^2 \alpha \cdot \cos^2 \alpha \cdot \cos^2 \beta\big)^{1/2} = \left(\sin^2\alpha \cdot \cos^2 \beta + \sin^2 \beta\right)^{1/2}.
\end{align*}

Combining the above leads to
\begin{align*}
\crc_2 &= \left\{c + r\cdot \cos\theta \cdot V_1 + r \cdot \sin \theta\cdot \frac{V_2}{\|V_2\|} \ :\ 0\le \theta< 2\pi \right\}
\end{align*}

As with $\crc_1$, we parametrize $\crc_2$ as
\begin{equation}\label{eq:SecondCircParam}
\gamma_2(t) = \left\{c + r\cdot  \frac{1-t^2}{1+t^2}\cdot V_1 + r\cdot \frac{2t}{1+t^2}\cdot \frac{V_2}{\|V_2\|} \ :\ t \in \RR \right\}.
\end{equation}
As before, the parametrization overlooks one point of $\crc_2$, corresponding to the case where $t\to \infty$.
If this point is in $\pts_2$, then we remove it from $\pts_2$.

\parag{Studying the distance function.}
To recap, in \eqref{eq:FirstCircParam} and \eqref{eq:SecondCircParam} we parametrized the two circles, possibly excluding one point from $\pts_1$ and another from $\pts_2$. 
We denote by $\gamma_{i,x}(s)$  the $x$-coordinate of $\gamma_i(s)$, and similarly for $\gamma_{i,y}(s)$ and $\gamma_{i,z}(s)$.
Let $\rho(s,t)$ be the square of the distance between $\gamma_1(s)$ and $\gamma_2(t)$.
That is,
\[ \rho(s,t) = (\gamma_{1,x}(s)-\gamma_{2,x}(t))^2 + (\gamma_{1,y}(s)-\gamma_{2,y}(t))^2+ (\gamma_{1,z}(s)-\gamma_{2,z}(t))^2. \]

The following is a bipartite variant of a result of Raz \cite{Raz20}.
The proof can be seen as a simple variant of the one in \cite{Raz20}.
Recall that a real function is \emph{analytic} if it has derivatives of every order and agrees with its Taylor series in a neighborhood of every point.

\begin{lemma} \label{le:SpecialorDist}
With the above definitions, at least one of the following holds:
\begin{itemize}
\item $D(\pts_1,\pts_2) = \Omega\left(\min\left\{m^{2/3}n^{2/3},m^2,n^2\right\}\right)$.
\item For $j\in \{1,2,3\}$, there exist an open interval $I_j \subset \RR$ and an analytic function $\varphi_j: I_j \to \RR$ with an analytic inverse, that satisfy the following.
For every $s\in I_1$ and $t\in I_2$, we have that $\rho(s,t) = \varphi_1(\varphi_2(s)+\varphi_3(t))$. \end{itemize}
\end{lemma}

Note that \eqref{eq:MultForm} is an example of $\rho(s,t) = \varphi_1(\varphi_2(s)+\varphi_3(t))$ in the perpendicular case.

\begin{proof}[Proof of Lemma \ref{le:SpecialorDist}.] 
The proof is based on the following result from Raz, Sharir, de Zeeuw \cite[Sections 2.1 and 2.3]{RSdZ16}.
See also Raz \cite[Lemma 2.4]{Raz20}.
 
\begin{theorem} \label{th:RSdZ}
Let $F \in \RR[x, y, z]$ be a constant-degree irreducible polynomial, such that none of the three first partial derivatives of $F$ is identically zero.
Then at least one of the following two cases holds.
\begin{enumerate}[label =(\roman*)]
    \item For all $A, B \subset \RR$ with $|A| = m$ and $|B| = n$, we have
    \begin{align*}
    \hspace{-11mm} \big|\big\{(a,a',b,b')\in A^2 \times B^2\ :\ \text{ exists } c\in \RR \text{ such that }  F(a&,b,c)=F(a',b',c)=0 \big\}\big| \\[2mm]
    &= O\left(m^{4/3}n^{4/3}+m^2+n^2\right).
    \end{align*}
    \item There exists a one-dimensional subvariety $Z^* \subset \vb(F)$ of complexity $O(1)$ such that every $p \in  \vb(F) \setminus Z^*$ satisfies the following:
    For each $j\in \{1,2,3\}$ there exist an open interval $I_j^{p} \subset \RR$ and a real analytic function  $\psi_j^{p}:I_j^{p}\to \RR$ with an analytic inverse such that $p \in I_1^{p}\times I_2^{p}\times I_3^{p}$ and for all $(x, y, z) \in I_1^{p}\times I_2^{p}\times I_3^{p}$ we have
    \[ (x, y, z) \in \vb(F) \quad \text{ if and only if } \quad \psi_1^{p}(x) + \psi_2^{p}(y) + \psi_3^{p}(z) = 0. \]
\end{enumerate}
\end{theorem}

We now briefly sketch the remainder of the proof of Lemma \ref{le:SpecialorDist}. We first construct a variety $U\subset \RR^3$ that in some way describes the distances between $C_1$ and $C_2$.
We then apply Theorem \ref{th:RSdZ} on  the polynomial $f\in \RR[x,y,z]$ that generates ${\bf I}(f)$. 
We show that, when we are in case (i) of the theorem, $D(\pts_1,\pts_2)$ is large.
Finally, we  show that case (ii) of the theorem implies $\rho(s,t) = \varphi_1(\varphi_2(s)+\varphi_3(t))$ as in the statement of Lemma \ref{le:SpecialorDist}.

It is not difficult to show that any circle in $\RR^3$ is the zero-set of a polynomial of degree four. 
Let $f_1,f_2\in \RR[x,y,z]$ be polynomials of degree four that satisfy $\crc_1=\vb(f_1)$ and $\crc_2=\vb(f_2)$.
Consider points $p=(p_x,p_y,p_z)\in \crc_1$ and $q=(q_x,q_y,q_z)\in \crc_2$ at a distance of $\delta\in \RR$ from each other.
Setting $\Delta=\delta^2$ leads to
\begin{align*}
f_1(p_x,p_y,p_z) &= 0, \\
f_2(q_x,q_y,q_z) &=0, \\
(p_x-q_x)^2+(p_y-q_y)^2&+(p_z-q_z)^2 = \Delta.
\end{align*}

Considering $p_x,p_y,p_z,q_x,q_y,q_z,\Delta$ as variables, the above system defines a variety $U\in\RR^7$ of complexity at most four.
At most two points $p\in \crc_1$ are contained in the axis of $\crc_2$.
For every other point $p\in \crc_1$ and squared distance $\Delta$, there are at most two values of $q$ that satisfy the above system.
This implies that $\dim U\le 2$.
Since $U$ is the union of infinitely many disjoint one-dimensional varieties (one for every $p\in C_1$), we conclude that $\dim U=2$.

Let $\pi:\RR^7 \to \RR^3$ be the projection defined as
\[ \pi(p_x,p_y,p_z,q_x,q_y,q_z,\Delta) = (p_x,q_x,\Delta). \]

Set $U_3 = \pi(U)$.
By Lemma \ref{le:projection}, this is a semi-algebraic set of dimension at most two and of complexity $O(1)$.
Since $\crc_1$ is a circle in the $xy$-plane, no three points on $\crc_1$ have the same $x$-coordinate.
The same holds for $\crc_2$, since it is a circle not contained in a plane parallel to the $yz$-plane.
Thus, the preimage of every point of $\pi(U)$ consists of at most four points of $U$.
This implies that $\dim U_3 =2$.

Consider the set
\[ Z = \left\{(\gamma_{1,x}(s),\gamma_{2,x}(t),\rho(s,t))\ : \ s,t\in \RR \right\}.\]
Note that $Z$ is equivalent to $U_3$ up to the two points that are not covered by the parametrizations $\gamma_1(s)$ and $\gamma_2(t)$.
These two missing points yield a one-dimensional constant-complexity semi-algebraic set $Z_0\subset U_3$ such that $U_3 = Z \cup Z_0$.

Let $F\in \RR[x,y,z]$ be a constant-degree polynomial satisfying $\overline{U_3} = \vb(F)$.
We claim that no first partial derivative of $F$ is identically zero.
Indeed, from the definition of $Z$ we note that $F$ must depend on all three coordinates.
We may thus apply Theorem \ref{th:RSdZ} with $F$.
We partition the remainder of the proof of Lemma \ref{le:SpecialorDist} according to the case of the theorem that holds.

\parag{The quadruples case.} First assume that the first case of Theorem \ref{th:RSdZ} holds (the case involving the number of quadruples).
Define
\[ Q = \big\{(a,a',b,'b)\in \pts_1^2 \times \pts_2^2\ :\ \text{ exists } c\in \RR \text{ such that }  F(a,b,c)=F(a',b',c)=0 \big\}. \]
By Theorem \ref{th:RSdZ}, we have that
\[ |Q| = O\left(m^{4/3}n^{4/3}+m^2+n^2\right).\]

For $\delta\in \RR$, set $m_\delta=|\{(a,b)\in\pts_1\times\pts_2\ :\ |ab|=\delta \}|$.
In other words, $m_\delta$ is the number of pairs in $\pts_1\times\pts_2$ at distance $\delta$.
Since every pair of $\pts_1\times\pts_2$ contributes to one $m_\delta$, we get that $\sum_\delta m_\delta = \Theta(mn)$.
The number of quadruples in $Q$ that satisfy $|ab|=|a'b'|=\delta$ is $m_\delta^2$.
Combining this observation with the Cauchy--Schwarz inequality gives us
\[ |Q| = \sum_{\delta}m_\delta^2 \ge \frac{(\sum_\delta m_\delta)^2}{D(\pts_1,\pts_2)} = \Theta\left(\frac{m^2n^2}{D(\pts_1,\pts_2)}\right). \]

 Combining the two above bounds for $|Q|$ gives
 \[ \frac{m^2n^2}{D(\pts_1,\pts_2)} = O\left(m^{4/3}n^{4/3}+m^2+n^2\right), \quad \text{ or } \quad D(\pts_1,\pts_2)=\Omega\left(\min\left\{m^{2/3}n^{2/3},m^2,n^2\right\}\right). \]
This completes the first case of the proof of Lemma \ref{le:SpecialorDist}.

\parag{The case where $F$ has a special form.} We now consider the second case of Theorem \ref{th:RSdZ} (the case stating that $F$ has a special form in an open neighborhood of most points).
As in the statement of the theorem, the exceptional set $Z_0^*$ is a variety of dimension one and complexity $O(1)$.
We consider a point $v=(x_0,y_0,\Delta_0)\in Z\setminus (Z_0 \cup Z^*)$.
Then for $j\in \{1,2,3\}$ there exist an open interval $I_j\subset \RR$ and a real-analytic function $\psi_j: I_j \to \RR$ with an analytic inverse that satisfy the following.
We have that $(x_0,y_0,\Delta_0)\in I_1\times I_2\times I_3$. Every $(x,y,\Delta)\in I_1\times I_2\times I_3$ satisfies
\begin{equation} \label{eq:SpecialForm}
(x, y, \Delta) \in \vb(F) \quad \text{ if and only if } \quad \psi_1(x) + \psi_2(y) + \psi_3(\Delta) = 0.
\end{equation}

Set $\psi'_1(x) = -\psi_1(x)$ and $\psi'_2(y) = -\psi_2(y)$.
We rewrite \eqref{eq:SpecialForm} as
\[(x, y, \Delta) \in \vb(F) \quad \text{ if and only if } \quad  \Delta = \psi_3^{-1}\left(\psi'_1(x) + \psi'_2(y)\right). \]

Recall that $\rho(s,t)=\Delta$, where we use the point parametrizations $\gamma_1(s)$ and $\gamma_2(t)$.
Thus, there exist open intervals $S,T\subset\RR$, such that for every $s\in S$ and $t\in T$ we have
\[\rho(s,t)=\Delta \quad  \text{ if and only if } \quad  \Delta = \psi_3^{-1}\big(\psi'_1(\gamma_{1,x}(s)) + \psi'_2(\gamma_{2,x}(t))\big). \]
That is, there exists an open neighborhood where
\[ \rho(s,t)=\psi_3^{-1}\big(\psi'_1(\gamma_{1,x}(s)) + \psi'_2(\gamma_{2,x}(t))\big). \]

Since $\gamma_{1,x}(s)$ and $\gamma_{2,x}(t)$ are rational functions in $t$ and $s$, they are analytic. 
By inspecting the definitions in \eqref{eq:FirstCircParam} and \eqref{eq:SecondCircParam}, we note that their inverses are solutions to quadratic equations. 
Thus, in a sufficiently small open neighborhood, we may remove the $\pm$ from the quadratic formula, to obtain analytic inverses.
This in turn implies that $\psi'_1 \circ \gamma_{1,x}(s)$ and $\psi'_2 \circ \gamma_{2,x}(t)$ are analytic with analytic inverses.
Setting $\varphi_1 = \psi_3^{-1}$, $\varphi_2 = \psi'_1 \circ \gamma_{1,x}(s)$, and $\varphi_3 = \psi'_2 \circ \gamma_{2,x}(t)$ concludes the second case of Lemma 5.1.
\end{proof}

To complete the proof of Theorem \ref{th:CircDD}(b), it remains to show that we cannot be in the second case of Lemma \ref{le:SpecialorDist}.
For this, we rely on the following \emph{derivative test} (see for example \cite{deZeeuw18}).
For a function $f$, we write $f_x= \frac{\partial f}{\partial x}$.

\begin{lemma} \label{le:DerivTest}
Let $f\in \RR[x,y]$ be twice differentiable with $f_y\not\equiv 0$.
Let $N$ be an open neighborhood in $\RR^2$.
If there are differentiable $\varphi_1,\varphi_2,\varphi_3 \in \RR[z]$ satisfying
\[ f(x,y) = \varphi_1(\varphi_2(x)+\varphi_3(y)) \ \text{ for every } (x,y)\in N  \]
then
\[ \frac{\partial^2(\log |f_x/f_y|)}{\partial x \partial y} \quad \text{is identically zero in $N$.} \]
\end{lemma} 
\begin{proof}
Assume that there exist $\varphi_1,\varphi_2,\varphi_3 \in \RR[z]$ as stated in the lemma.
Then
\begin{align*}
f_x(x,y) = \frac{\partial \varphi_1}{\partial x}(x,y) \cdot \frac{\partial \varphi_2}{\partial x}(x) \quad \text{ and } \quad
f_y(x,y) = \frac{\partial \varphi_1}{\partial y}(x,y) \cdot \frac{\partial \varphi_3}{\partial y}(y). 
\end{align*}

Combining the above gives
\[\log\left|f_x / f_y\right|  = \log\left|\frac{\partial \varphi_2}{\partial x}(x)/\frac{\partial \varphi_3}{\partial y}(y)\right| = \log\left|\frac{\partial \varphi_2}{\partial x}(x)\right|-\log \left|\frac{\partial \varphi_3}{\partial y}(y)\right|.\]
Differentiating with respect to both $x$ and $y$ gives 0, which completes the proof. 
\end{proof}

Recall that our goal is to show that the second case of Lemma \ref{le:SpecialorDist} cannot happen. 
By Lemma \ref{le:DerivTest}, it suffices to show that 
\begin{equation}\label{eq:derivative_test}
    g(s,t) = \frac{\partial^2(\log |\rho_t/\rho_s|)}{\partial s \partial t}
\end{equation}
is not identically zero in any open neighborhood of $\RR^2$. 
Since $\rho(s,t)$ is a rational function in $s$ and $t$, so is $g(s,t)$. 
Thus, the $g(s,t)$ is identically zero everywhere if and only if the numerator is identically zero everywhere. That is, if every coefficient in the numerator is 0. 

Since \eqref{eq:derivative_test} includes a logarithm and an absolute value, how can we say that $g(s,t)$ is rational?
To see that, we set $h(s, t) = \rho_t/\rho_s$. 
Since $\rho(s, t)$ is a rational function in $s$ and $t$, so is $h(s, t)$.
For each point $(s, t)$ that satisifies $\rho_s(s,t)\neq 0$ and $\rho_t(s,t)\neq 0$, there exists an open neighbourhood where $|h(s, t)| = c \cdot h(s, t)$, where $c\in \{-1,1\}$. In either case, we have that %
\[ \frac{\partial}{\partial t} \log |h(s, t)| = \frac{1}{c\cdot h(s, t)} \cdot \frac{\partial}{\partial t} (c \cdot h(s, t)) = \frac{h_t(s, t)}{h(s, t)}. \]
This implies that, at any point where both $\rho_t$ and  $\rho_s$ are nonzero, we have that $g(s, t) = \frac{\partial}{\partial s} \left(\frac{h_t(s, t)}{h(s, t)}\right)$.
This is indeed a rational function.

Since $\rho(s, t)$ depends on parameters $p, q, r, \alpha$ and $\beta$, the expression $g(s,t)$ is too large to compute by hand. 
Instead, we used Mathematica \cite{Wolfram} to compute it. 
Our code can be found in Appendix \ref{app:code}.
Since $g(s, t)$ is a rational function, it suffices to determine the parameters $p, q, r, \alpha$ and $\beta$ for which the coefficients of some of the monomials in the numerator do not simultaneously vanish.
We do so in Appendix \ref{app:analysis}.
In the same appendix, we also address the special cases that were ignored in the above analysis.
\end{proof}

\appendix

\section{The Mathematica code} \label{app:code}

In this appendix, we describe the Mathematica program that was used in the proof of Theorem \ref{th:CircDD}(b).
Listing \ref{list:code} contains the code that is used for the general case of the proof.
Lines 1--7 define the parametrizations described in \eqref{eq:FirstCircParam} and \eqref{eq:SecondCircParam}.
Lines 8--9 define the function $\rho(s,t)$.
Line 10 is the derivative test of Lemma \ref{le:DerivTest}.
Finally, line 11 shows the coefficient of a specific term.

\begin{lstlisting}[language=Mathematica,caption={The Mathematica code.\label{list:code}}] 
v1 = {Cos[$\displaystyle \alpha$], 0, Sin[$\displaystyle \alpha$]};
v2 = {Cos[$\displaystyle \beta$] Sin[$\displaystyle \alpha$]$\displaystyle^2$, Sin[$\displaystyle \beta$], -Cos[$\displaystyle \alpha$] Cos[$\displaystyle \beta$] Sin[$\displaystyle \alpha$]};
v2norm = Sqrt[(v2[[1]]$\displaystyle ^\mathsf{2}$ + v2[[2]]$\displaystyle ^\mathsf{2}$ + v2[[3]]$\displaystyle ^\mathsf{2}$)];
c1 = {2 s/(1 + s$\displaystyle ^\mathsf{2}$), (1 - s$\displaystyle ^\mathsf{2}$)/(1 + s$\displaystyle ^\mathsf{2}$), 0};
c2 = {p + r (1 - t$\displaystyle ^\mathsf{2}$) v1[[1]]/(1 + t$\displaystyle ^\mathsf{2}$) + 2 r t v2[[1]]/((1 + t$\displaystyle ^\mathsf{2}$) v2norm),
      r (1 - t$\displaystyle ^\mathsf{2}$) v1[[2]]/(1 + t$\displaystyle ^\mathsf{2}$) + 2 r t v2[[2]]/((1 + t$\displaystyle ^\mathsf{2}$) v2norm),
      q + r (1 - t$\displaystyle ^\mathsf{2}$) v1[[3]]/(1 + t$\displaystyle ^\mathsf{2}$) + 2 r t v2[[3]]/((1 + t$\displaystyle ^\mathsf{2}$) v2norm)};
dist[s_, t_] :=
    Together[Expand[(c1[[1]] - c2[[1]])$\displaystyle ^\mathsf{2}$ + (c1[[2]] -c2[[2]])$\displaystyle ^\mathsf{2}$ + (c1[[3]] - c2[[3]])$\displaystyle ^\mathsf{2}$]];
RhoTest = Numerator[Together[D[D[Log[D[dist[s, t], s]/D[dist[s, t], t]], s], t]]];
Coefficient[RhoTest, s$\displaystyle ^\mathsf{5}$ t] // Factor
  \end{lstlisting}

The code of Listing \ref{list:code} leads to expressions somewhat more involved than the ones stated in Section \ref{sec:ProofB}.
For example, the coefficient of $s^5t$ produced by Mathematica is
\begin{align*}
96 r (p + r \cos\alpha) \cos\beta \sin^2\alpha (-r &+
   p \cos\alpha + r \cos^2\alpha + q \sin\alpha +
   r \sin^2 \alpha) \\
   &\cdot \sin\beta (\cos^2\alpha \cos^2\beta
\sin^2\alpha + \cos^2\beta \sin^4\alpha + \sin^2\beta)
\end{align*}

The above expression can be simplified by noting that
\begin{align*}
-r + p \cos\alpha + r \cos^2\alpha + q \sin\alpha + r \sin^2 \alpha &= p \cos\alpha + q \sin\alpha, \text{ and }\\[2mm]
\cos^2\alpha \cos^2\beta \sin^2\alpha + \cos^2\beta \sin^4\alpha + \sin^2\beta &= \cos^2\beta \sin^2\alpha + \sin^2\beta.
\end{align*}

Looking at Lemma \ref{le:DerivTest}, it may seem as if we forgot to include the absolute value in Listing \ref{list:code}.
For some $f\in \RR[x]$, consider
\[ \frac{\partial}{\partial x} \log |f(x)| = \begin{cases}
\frac{\partial}{\partial x} \log f(x) \qquad & \text{ if } f(x)> 0, \\
\frac{\partial}{\partial x} \log \left(-f(x)\right) \qquad & \text{ if } f(x)< 0.
\end{cases}
\]
In either case the derivative is $f'(x)/f(x)$.
We may thus ignore the absolute value in Lemma \ref{le:DerivTest}.

The special cases at the end of Section \ref{sec:ProofB} require minor changes in the code from Listing \ref{list:code}.
For example, the case of $\cos \beta=0$ is obtained by changing lines 2--3 to $v_2=(0,1,0)$.
The case of $\crc_2$ being centered at the origin is obtained by removing ``p+'' for line 5 and ``q+'' from line 7.

\section{Analysis of the derivative test} \label{app:analysis}

In this appendix, we complete the proof of Theorem \ref{th:CircDD}(b) by showing that $g(s, t)$ as defined in \eqref{eq:derivative_test} cannot vanish everywhere, and also address the special cases. To do so, we note that $g(s, t)$ is a rational function, and it suffices to show that there exist monomials in the numerator whose coefficients cannot simultaneously vanish.
We consider several such monomials:

\begingroup\abovedisplayskip=4pt \belowdisplayskip=4pt

\begin{itemize}[noitemsep,topsep=1pt]
\item The coefficient of $s^5t$ is 
\[ 96 r (p + r \cos\alpha) \cos\beta \sin^2\alpha (p \cos\alpha + q \sin\alpha) \sin\beta (\cos^2\beta \sin^2\alpha + \sin^2\beta).\]
\item The coefficient of $s^5t^9$ is 
\[ 96 r (p - r \cos\alpha) \cos\beta \sin^2 \alpha (p \cos\alpha + q \sin\alpha) \sin\beta (\cos^2\beta \sin^2\alpha + \sin^2\beta). \]
\item The coefficient of $s^3t^9$ is
\begin{align*} 
64 r &(p - r \cos\alpha) \cos\beta \sin\alpha \sin\beta (\cos^2\beta \sin^2\alpha + \sin^2\beta) \\[2mm]
&\hspace{60mm} \cdot (-2 q \cos^2\alpha + p \cos\alpha\sin\alpha - q \sin^2\alpha ).
\end{align*}
\item The coefficient of $s^3t$ is
\begin{align*}  
64 r &(p +   r \cos\alpha) \cos\beta \sin\alpha \sin\beta (\cos^2\beta \sin^2\alpha + \sin^2\beta)\\[2mm]
&\hspace{60mm} \cdot (-2 q \cos^2\alpha + p \cos\alpha\sin\alpha - q \sin^2\alpha ) .
\end{align*}
\end{itemize}

\endgroup

By definition, $r\neq 0$. Recall that $H$ (the plane containing $C_2$) is not parallel to the $xy$-plane or the $xz$-plane.
Since we also assumed $H$ does not contain lines parallel to the $x$-axis, we have that  $\sin\alpha \neq 0$ and $\sin\beta \neq 0$. Indeed, this is easy to see when recalling that $V'_1$ and $V'_2$ from \eqref{eq:V1V2} span the directions of $H$.
This implies that $(\cos^2\beta \sin^2\alpha + \sin^2\beta)\neq 0$.
Since we assume that $H$ is not parallel to the $xz$-plane, we have that $\sin\beta\neq 0$.
We now also assume that the center of $\crc_2$ is not the origin, that $\cos \alpha \neq 0$, and that $\cos \beta \neq 0$.
These special cases are  addressed after completing the general cases. 

Since $\cos \alpha \neq 0$, it is not possible for $(p - r \cos\alpha)$ and $(p + r \cos\alpha)$ to be zero simultaneously.
Thus, the only way for all of the four above coefficients to equal zero is to have both 
\[ p \cos\alpha + q \sin\alpha=0 \quad \text{ and } \quad  -2 q \cos^2\alpha + p \cos\alpha\sin\alpha - q \sin^2\alpha=0. \]

We rearrange the first equation as $p \cos\alpha = -q \sin\alpha$. 
Plugging this into the second equation gives
\[ 0 = -2 q \cos^2\alpha + p \cos\alpha\sin\alpha - q \sin^2\alpha = -2 q \cos^2\alpha  - 2q \sin^2\alpha = -2q. \]
That is, $q=0$. 
We then have $p \cos\alpha = -q \sin\alpha =0$, or $p=0$.
This contradicts the assumption that $C_2$ is not centered at the origin.

By the above, it is impossible for the four above coefficients to be zero simultaneously. 
This implies that $g(s,t)$ is not identically zero in any open neighborhood $N\subset\RR^2$.
By Lemma \ref{le:DerivTest}, we get that $\rho(s,t)$ cannot be rewritten as $\varphi_1(\varphi_2(s)+\varphi_3(t))$ for every $(x,y)\in N$.
Then, Lemma \ref{le:SpecialorDist} implies the assertion of the theorem. 

\parag{The special cases.} 
In the above proof, we assume: \begin{itemize}[noitemsep,topsep=1pt]
    \item The plane $H$ is not parallel to the $xy$-plane, to the $xz$-plane, or to the $yz$-plane.
    \item $H$ does not contain lines parallel to the $x$-axis.
    \item The center of $\crc_2$ is not the origin.
    \item $\cos\beta \neq 0$.
    \item $\cos\alpha \neq 0$.
\end{itemize}
We now address each of these special cases, in the above order.

We first consider the case where $H$ is parallel to the $xy$-plane.
As discussed in the introduction, if $H$ is the $xy$-plane and $\crc_1$ and $\crc_2$ are not concentric, then $D(\pts_1,\pts_2) = \Omega\left(\min\left\{m^{2/3}n^{2/3},m^2,n^2\right\}\right)$.
When $H$ is parallel to the $xy$-plane, we can combine the above with Lemma \ref{le:parallelPlanes}, to obtain the following. If $\crc_1$ and $\crc_2$ are not aligned then $D(\pts_1,\pts_2) = \Omega\left(\min\left\{m^{2/3}n^{2/3},m^2,n^2\right\}\right)$.

Next consider the case where $H$ is parallel to the $xz$-plane. 
In this case, we can rewrite $V_1 = (1,0,0)$ and $V_2 = (0,0,1)$, which significantly simplifies \eqref{eq:SecondCircParam}.
The Mathematica program then implies that the numerator of $g(s,t)$ is $4 q (1 + t^2) (-1 + s^2)$.
This expression is not identically zero unless $q=0$.
If $q=0$ then the center of $\crc_2$ is on the $x$-axis and $H$ contains the origin.
That is, if $q=0$ then $\crc_1$ and $\crc_2$ are perpendicular.
We conclude that either $D(\pts_1,\pts_2) = \Omega\left(\min\left\{m^{2/3}n^{2/3},m^2,n^2\right\}\right)$ or the circles are perpendicular.

We move to the case where $H$ is parallel to the $yz$-plane.
We rewrite $V_1 = (0,1,0)$ and $V_2 = (0,0,1)$.
In this case, the coefficient of $s^2$ in the numerator of $g(s,t)$ is $4pr$.
The coefficient of $s^3t^4$ is $16q(3p^2-8r^2)$.
For both of these coefficients to be zero, we must have $p=q=0$, implying that the two circles are perpendicular. 
Once again, either $D(\pts_1,\pts_2) = \Omega\left(\min\left\{m^{2/3}n^{2/3},m^2,n^2\right\}\right)$ or the circles are perpendicular.

We next assume that $H$ contains lines parallel to the $x$-axis.
In this case, we can write $V_1 = (1,0,0)$ and $V_2=(0,\cos\gamma,\sin\gamma)$.
Since $H$ is not parallel to the $xy$-plane and $xz$-plane, we have that $\sin \gamma \neq 0$ and $\cos \gamma \neq 0$.
We consider the following coefficients of $g(s,t)$:
\begin{itemize}[noitemsep,topsep=1pt]
\item The coefficient of $s$ (with no $t$) is $-8p  (p+r)^2 \cos \gamma$.
\item The coefficient of $st^6$ is $-8p(p-r)^2 \cos\gamma$. 
\item The coefficient of $s^3 t^5$ is $-128q(p-r)r\cos\gamma \sin \gamma$.
\end{itemize}

Since $p+r$ and $p-r$ cannot simultaneously be zero, the first two coefficients imply that $p=0$.
Since $r>0$, we get that $p-r\neq 0$.
Then, the third coefficient implies that $q=0$.
That is, we are in the special case where $V_1 = (1,0,0)$, $V_2=(0,\cos\gamma,\sin\gamma)$, and $p=q=0$. In this case, the coefficient of $s^6t^2$ is $-4 \cos \gamma\sin^2\gamma$.
This coefficient is never zero, which completes the case of lines parallel to the $x$-axis.

Note that $\sin \alpha \neq 0$, since otherwise $H$ is parallel to the $xy$-plane. 
Similarly, $\sin\beta\neq 0$, since otherwise $H$ is parallel to the $xz$-plane.
We next consider the case where $\crc_2$ is centered at the origin.
That is, we set $p=q=0$.
In this case the coefficient of $t^3$ (with no $s$ factor) in the numerator of $g(s,t)$ is 
\[  -32 \cos^2 \alpha \cos \beta \sin^2 \alpha \sin\beta (\cos^2\beta \sin^2\alpha + \sin^2\beta). \]

For this expression to be zero, we must have either $\cos \alpha=0$ or $\cos \beta=0$.
If $\cos\alpha=0$, then $H$ is perpendicular to the $xy$-plane, so the two circles are perpendicular. 
If $\cos\beta=0$, then $V_2=(0,1,0)$.
Running the program once again gives that the coefficient of $s^6 t^6$ is $-8 \cos\alpha\sin^2\alpha$.
For this coefficient to vanish, we again require $\cos\alpha=0$, so the circles are again perpendicular.

Consider the case where $\cos \beta=0$. 
In this case $\sin\beta =1$, so $V_2 = (0,1,0)$.
As before, we run the Mathematica program to find coefficients in the numerator of $g(s,t)$.
The coefficient of $s^5t^3$ is $-128 p r \cos\alpha$.
For this coefficient to vanish, we require either $p=0$ or $\cos\alpha=0$.
Consider the case where $\cos\alpha=0$.
In this case, the coefficient of $t^{10}$ (and no factor of $s$) is $-4pr \sin^2\alpha$.
Since $\sin\alpha =1$, we get that $p=0$.
That is, in either case we have that $p=0$.

We continue the case where $\cos \beta=0$ and $p=0$.
We may assume that $q\neq 0$, since we already handled the case where the center of $\crc_2$ is the origin.
The coefficient of $t^9$ is $-16 r^2q \cos\alpha \sin\alpha$.
This coefficient vanishes if and only if $\cos \alpha=0$. 
In this case $V_1=(0,0,1)$ and running the program again leads to the numerator being $8 q (1 + t^2) s$.
Since $q\neq 0$, this numerator does not vanish identically. 

Finally, we assume that $\cos \alpha=0$.
In this case, $V_1=(0,0,1)$, so $H$ is perpendicular to the $xy$-plane.
We may assume that $\cos\beta\neq 0$, since this is the case where $H$ is parallel to the $yz$-plane.
As usual, we consider coefficients of terms in the numerator of $g(s,t)$.
The coefficient of $s^3t^3$ is $-384 p q r \cos\beta \sin\beta$.
For this coefficient to vanish, we require either $p=0$ or $q=0$.
When $p=0$, the numerator becomes $4 q (1 + t^2) (-\cos\beta + s^2 \cos\beta + 2 s \sin\alpha)$.
When $q=0$, the numerator becomes $-4 p r (-1 + t^2) (1 + s^2) \sin\beta$.
In either case, the numerators is zero if and only if $p=q=0$, which is a case we already handled.
\end{document}